\documentclass{amsart} 

\usepackage[english,french]{babel}
\usepackage[latin1]{inputenc}

\usepackage{amsmath}
\usepackage{amsthm}
\usepackage{amssymb}
\usepackage{tikz}

\usepackage{imakeidx} 
\usepackage{appendix}
\usepackage{tikz-cd}
\usepackage{verbatim}
\usepackage{enumitem}
\usepackage{multicol}

\usepackage[backref=page]{hyperref}
\usepackage{cleveref}

\hypersetup{colorlinks=true,linkcolor=blue,anchorcolor=blue,citecolor=blue}

\usepackage{adjustbox}

\usepackage{mathtools}

\newtheorem{thm}{Th\'eor\`eme}[section]
\newtheorem{prop}[thm]{Proposition}
\newtheorem{lemma}[thm]{Lemme}
\newtheorem{cor}[thm]{Corollaire}
\newtheorem{conj}[thm]{Conjecture}

\theoremstyle{definition}

\theoremstyle{remark}

\newtheorem{rmk}[thm]{Remarque}
\numberwithin{equation}{section}

\newcommand{\Q}{\mathbb Q}
\newcommand{\F}{\mathbb F}
\newcommand{\C}{\mathbb C}
\newcommand{\Z}{\mathbb Z}
\newcommand{\G}{\mathbb G}
\newcommand{\g}{\mathfrak g}
\newcommand{\h}{\mathfrak h}
\renewcommand{\P}{\mathbb P}
\newcommand{\Spec}{\operatorname{Spec}}

\newcommand{\mc}[1]{\mathcal{#1}}
\newcommand{\cl}{\overline}
\newcommand{\set}[1]{\left\{#1\right\}}
\renewcommand{\phi}{\varphi}

\newcommand{\on}[1]{\operatorname{#1}}

\title[Produit d'une courbe et d'une surface]{Sur la conjecture de Tate enti\`ere pour le produit d'une courbe et d'une surface $CH_{0}$-triviale  sur un corps fini}

\author{Jean-Louis Colliot-Th\'el\`ene et Federico Scavia}

\address{Universit\'e Paris-Saclay, CNRS, Laboratoire de math\'ematiques d'Orsay, 91405, Orsay, France.}
\email{jlct@math.u-psud.fr}

\address{Department of Mathematics\\
	University of British Columbia\\
	Vancouver, BC V6T 1Z2\\Canada}

\email{scavia@math.ubc.ca}

\date{20 juillet 2021}


\begin{document}
	\maketitle
	\hypersetup{backref=true}
	
	\selectlanguage{english}	
	\begin{abstract}
		We investigate 
		a strong version of
		the integral Tate conjecture for 1-cycles on the product
		of a curve and a surface over a finite field,
		under the assumption that the surface is geometrically $CH_0$-trivial.
		By this we mean that over any algebraically closed field extension, the degree map on the
		zero-dimensional Chow group of the surface is an isomorphism. This applies
		to Enriques surfaces. When the N\'eron-Severi group has no torsion,
		we recover earlier results of A. Pirutka. 
		The results rely on a detailed study of the third unramified cohomology group of specific products of varieties.
	\end{abstract}

	\selectlanguage{french}
	\begin{abstract}
		Nous \'etudions une forme forte de 	la conjecture  de Tate enti\`ere pour les 1-cycles sur le produit
		d'une courbe et d'une surface sur un corps fini, sous l'hypoth\`ese que la surface est g\'eom\'etriquement 
		$ CH_0 $-triviale.
		Nous entendons par cela que, sur toute extension de 
		corps
		alg\'ebriquement clos, la fl\`eche degr\'e sur le groupe de Chow de dimension z\'ero de la surface est un isomorphisme. Cela s'applique aux surfaces d'Enriques. Lorsque le groupe de N\'eron-Severi n'a pas de torsion, nous 
		retrouvons
		des r\'esultats ant\'erieurs de A. Pirutka. 
		Le travail implique
		l'\'etude d\'etaill\'ee du troisi\`eme groupe de cohomologie non ramifi\'ee
		pour les produits de vari\'et\'es consid\'er\'es.
	\end{abstract}
	
	\section{Introduction}
	Soient $\F$ un corps fini de caract\'eristique $p$, $\cl{\F}$ une cl\^oture alg\'ebrique de $\F$, et soit $G$ le groupe de Galois absolu $\on{Gal}(\cl{\F}/\F)$ et $\ell\neq p$ un nombre premier. Soit $X$ une $\F$-vari\'et\'e projective, lisse et g\'eom\'etriquement connexe, de dimension $d$, et soit $\cl{X}:=X\times_{\F}\cl{\F}$. Si $i\geq 0$ est un entier et $\ell\neq p$ est un nombre premier, la conjecture de Tate pour les cycles de codimension $i$ en cohomologie $\ell$-adique pr\'edit que les applications cycle
	\begin{equation}\label{tate1}
		CH^i(X)	\otimes_{\Z}{\Q_{\ell}}\to H^{2i}(\cl{X},\Q_{\ell}(i))^G, 
	\end{equation}
	\begin{equation}\label{tate2}
		CH^i(\cl{X})\otimes_{\Z}{\Q_{\ell}}\to H^{2i}(\cl{X},\Q_{\ell}(i))^{(1)} 
	\end{equation}
	et
	\begin{equation}\label{tate3}
		CH^i(X)\otimes_{\Z}{\Q_{\ell}}\to H^{2i}(X,\Q_{\ell}(i))	
	\end{equation}
	sont surjectives. Dans (\ref{tate2}), si $M$ est un $G$-module, on note par $M^{(1)}\subset M$ le sous-groupe form\'e des \'el\'ements dont le stabilisateur est un sous-groupe ouvert de $G$. En 
	fait
	ces trois versions de la conjecture de Tate pour toute extension finie de $\F$ sont \'equivalentes entre eux: l'\'equivalence entre la surjectivit\'e de (\ref{tate1})
	et de (\ref{tate2}) suit par un argument de restriction-corestriction, et celle entre la surjectivit\'e de (\ref{tate1}) et de (\ref{tate3}) utilise les conjectures de Weil. 
	
	On s'int\'eresse ici aux variantes enti\`eres de la conjecture de Tate, obtenues en rempla\c{c}ant partout $\Q_{\ell}$ par $\Z_{\ell}$. On ne s'attend pas \`a ce que ces variantes soient vraies en toute g\'en\'eralit\'e, mais on a 
	des raisons 
	d'esp\'erer dans le cas $i=d-1$, c'est \`a dire pour les $1$-cycles; voir \cite[\S 2]{ctszamuely}. Les questions d'int\'er\^et sont donc les suivantes: Est-ce que les applications 
	\begin{equation}\label{tate-int1}\tag{1.1'}
		CH^{d-1}(X)	\otimes_{\Z}{\Z_{\ell}}\to H^{2d-2}(\cl{X},\Z_{\ell}(d-1))^G,	
	\end{equation}
	\begin{equation}\label{tate-int2}\tag{1.2'}
		CH^{d-1}(\cl{X})\otimes_{\Z}{\Z_{\ell}}\to H^{2d-2}(\cl{X},\Z_{\ell}(d-1))^{(1)}	
	\end{equation}
	ou 
	\begin{equation}\label{tate-int3}\tag{1.3'}
		CH^{d-1}(X)\otimes_{\Z}{\Z_{\ell}}\to H^{2d-2}(X,\Z_{\ell}(d-1))	
	\end{equation}
	sont surjectives? On ne sait pas si (\ref{tate-int1}), (\ref{tate-int2}) et (\ref{tate-int3}) sont surjectives en g\'en\'eral, ou m\^eme si leurs surjectivit\'es sont \'equivalentes entre elles. 
	
	La surjectivit\'e de  (\ref{tate-int3}) pour $d=2$ \'equivaut \`a la conjecture de Tate initiale (surjectivit\'e de (\ref{tate3}))
	pour $d=2$ et $i=1$.
	Comme l'on voit \`a  la proposition \ref{dim3suffit}, pour \'etablir
	la surjectivit\'e de
	(\ref{tate-int3}) pour tout $d\geq 3$ (\Cref{cycle0} ci-dessous), il suffit de consid\'erer le cas $d=3$.  	
	
	Par ailleurs
	un th\'eor\`eme  bien connu de Schoen  (\Cref{schoen}) affirme que si la conjecture de Tate vaut pour les diviseurs sur les surfaces sur les corps finis (donc $d=2$ et $i=1$), alors l'application  (\ref{tate-int2}) est surjective pour tout $d$.

	Pour un solide (c'est-\`a-dire, une vari\'et\'e de dimension trois) projectif et lisse sur $\C$, la conjecture de Hodge enti\`ere  pour les $1$-cycles ne vaut pas; voir les exemples rappel\'es dans \cite{ctvoisin}. Elle est m\^eme n\'egative pour des solides \`a la g\'eom\'etrie relativement simple, comme le produit d'une courbe elliptique et d'une surface d'Enriques ``tr\`es g\'en\'erales'', comme l'on montr\'e r\'ecemment Benoist et Ottem  \cite{benoist2018failure}, voir aussi \cite{colliot2019cohomologie}. La situation sur les corps finis est donc (conjecturalement) tr\`es diff\'erente de celle sur $\C$. En effet, les contre-exemples de \cite{benoist2018failure} sont obtenus par une m\'ethode de sp\'ecialisation, qui n'est pas reproductible sur un corps fini. C'est aussi les cas pour les contre-exemples de Koll\'ar \cite[\S 5.3]{ctvoisin}: sous l'hypoth\`ese que la conjecture de Tate sur les diviseurs sur les surfaces sur un corps fini vaut, le th\'eor\`eme de Schoen implique qu'il n'existe pas d'exemple \`a la Koll\'ar sur $\F$ (ni m\^eme sur $\cl{\F}$); voir \cite[\S 6]{ctszamuely}.

	Il est donc naturel d'\'etudier la surjectivit\'e de (\ref{tate-int1}), (\ref{tate-int2}) et (\ref{tate-int3}) pour les produits d'une courbe elliptique et d'une surface d'Enriques sur $\F$. Nos r\'esultats principaux ne concernent que la surjectivit\'e de (\ref{tate-int3}), 
	qui n'est connue que dans tr\`es peu de cas.
	Le th\'eor\`eme de Schoen ne dit rien sur (\ref{tate-int3}), m\^eme sous la conjecture de Tate pour les diviseurs sur les surfaces. 
	Par ailleurs la surjectivit\'e de (\ref{tate-int3}) est un \'enonc\'e plus fort que  la surjectivit\'e de (\ref{tate-int1}); voir la remarque \ref{tateclassique} (ii).

	Pour $X/\F$ de dimension $d=3,$ notre point de d\'epart est le lien entre 
	le conoyau de l'application 
	(\ref{tate-int3}) et le groupe de cohomologie non-ramifi\'ee $H^3_{\on{nr}}(\F(X)/\F,\Q_{\ell}/\Z_{\ell}(2))$, cas particulier d'un \'enonc\'e pour les cycles de codimension 2
	\'etabli  par B. Kahn \cite{kahn2012classes} et
	par B. Kahn et le premier auteur  \cite{colliot2013cycles}. 
	
	Soient $k$ un corps  d'exposant caract\'eristique $p$ et $X$ une $k$-vari\'et\'e projective, lisse, et g\'eom\'etriquement connexe et soit $k(X)$ le corps des fonctions rationnelles de $X$.
	Pour chaque nombre premier $\ell\neq p$ et chaque entier $n\geq 1$, on a un groupe de cohomologie non-ramifi\'ee  \cite[Thm. 4.1.1]{colliot1995birational}
	$$H^n_{\on{nr}}(k(X)/k,\Q_{\ell}/\Z_{\ell}(n-1))
	\subset   H^n(k(X),\Q_{\ell}/\Z_{\ell}(n-1))$$ qui est un invariant  $k$-birationnel (et m\^{e}me stablement $k$-birationnel) de $X$.	
	Pour $n=2$, c'est la partie $\ell$-primaire du groupe de Brauer de $X$.
	Pour $n=3$, ce groupe intervient dans des questions de rationalit\'e, de descente galoisienne des classes de cycles de codimension $2$ modulo \'equivalence rationnelle, et aussi dans  l'\'etude des z\'ero-cycles des vari\'et\'es d\'efinies sur les corps globaux; voir \cite{colliot2013cycles}.  
	
	Dans le cas o\`u $k=\F$ est un corps fini, on sait que $H^3_{\on{nr}}(\F(X)/\F,\Q_{\ell}/\Z_{\ell}(2))$ est nul si $\dim X\leq 2$: c'est  trivial  si $\dim (X)=1$, et a \'et\'e \'etabli dans \cite[Remarque 2, p. 790]{colliot1983torsion} par Sansuc et Soul\'e et le premier auteur si $\dim (X)=2$, et aussi   par Kato \cite[Thm. 0.7, Corollaire]{katoclf}.
	
	Dans \cite{pirutka2011groupe},  Pirutka a construit des exemples de $\F$-vari\'et\'es projectives, lisses et g\'eom\'etri\-que\-ment rationnelles $X$ de  dimension $5$, et donc aussi de toute dimension $\geq 5$, avec  $H^3_{\on{nr}}(\F(X)/\F,\Q_{\ell}/\Z_{\ell}(2))\neq 0$.  Pour $\on{dim}(X)=3$ et pour $\on{dim}(X)=4$,  on ne sait pas s'il existe des vari\'et\'es projectives et lisses avec  $H^3_{\on{nr}}(\F(X)/\F,\Q_{\ell}/\Z_{\ell}(2))\neq 0$. 
	
	\medskip
	
	Dans \cite[Question 5.4]{colliot2013cycles}, Kahn et le premier auteur ont demand\'e: Est-ce que
	\begin{equation}\label{nr-trivial}
		H^3_{\on{nr}}(\F(X)/\F,\Q_{\ell}/\Z_{\ell}(2))=0	
	\end{equation}
	pour tout solide $X$ sur $\F$ et tout $\ell\neq p$? 
	Comme on explique au \S \ref{cyclesTate} (\Cref{coktatefini}  et \Cref{KCTK}),
	si la r\'eponse  \`a cette question est affirmative,  
	et si la conjecture de Tate vaut pour les diviseurs 
	sur
	le solide
	$X$,
	alors (\ref{tate-int3}) est surjective. 
	 . 
	
	Kahn et le premier auteur  ont conjectur\'e que la r\'eponse \`a (\ref{nr-trivial}) est affirmative si le solide $X$ est g\'eom\'etriquement unir\'egl\'e; voir \cite[Conjecture 5.7]{colliot2013cycles}.  L'annulation du groupe	 $H^3_{\on{nr}}(\F(X)/\F,\Q_{\ell}/\Z_{\ell}(2))$
	a \'et\'e \'etablie par  Parimala et Suresh \cite{parimala2016degree} si $X$ est
	un fibr\'e en coniques au dessus d'une surface. 
	Elle a aussi \'et\'e \'etablie  par  Pirutka \cite[Th\'eor\`eme 1.1]{pirutka2016cohomologie} lorsque le solide $X$
	est le produit d'une courbe $C$ et d'une surface $S$ g\'eom\'etriquement $CH_{0}$-triviale
	telle que $H^1(S,\mc{O}_S)=0$, sous l'hypoth\`ese suppl\'ementaire que le groupe de N\'eron-Severi g\'eom\'etrique
	$ {\rm NS}(\cl{S})$ 
	de $S$
	n'a pas de torsion, ce qui est le cas par exemple lorsque $S$ est une surface g\'eom\'etriquement rationnelle. On dit qu'une surface $S$ est \emph{g\'eom\'etriquement $CH_{0}$-triviale} si pour toute extension alg\'e\-bri\-quement close $\Omega$ de $\F$, le degr\'e \[CH_{0}(S_{\Omega})\otimes\Q \to \Q\] est un isomorphisme. On donne ci-dessous des rappels sur cette hypoth\`ese, avec ses cons\'equences remarquables. 
	Les surfaces g\'eom\'etriquement $CH_{0}$-triviales font l'objet de nombreuses investigations,	
	voir par exemple \cite{gorchinskiy2013geometric}.
	Pour $p\neq 2$, toute surface d'Enriques est g\'eom\'etriquement $CH_{0}$-triviale.

	Nous arrivons maintenant \`a nos r\'esultats principaux, qui soutiennent une r\'eponse affirmative \`a (\ref{tate-int3}) et (\ref{nr-trivial}) dans certains cas particuliers. Comme on a d\'ej\`a mentionn\'e, par la proposition \ref{dim3suffit} le cas plus int\'eressant de (\ref{tate-int3}) est $d=3$. Motiv\'es par \cite{benoist2018failure} et \cite{pirutka2016cohomologie}, nous consid\'erons des solides $X$ de la forme suivante. Soient $C$ et $S$ deux vari\'et\'es g\'eom\'etriquement connexes, projectives et lisses sur $\F$, de dimension $1$ et $2$ respectivement, $J(C)$ la jacobienne de la courbe $C$, et $X:=C\times_{\F} S$. 	
	
	Nous \'etablissons le th\'eor\`eme g\'en\'eral suivant.

	\begin{thm}\label{mainthm1}
		On suppose que la surface $S$ est g\'eom\'etriquement $CH_{0}$-triviale et que $\ell$ ne divise pas l'ordre de $\on{NS}(\cl{S})_{\on{tors}}$. Alors $H^3_{\on{nr}}(\F(X),\Q_{\ell}/\Z_{\ell}(2))=0$, et l'application cycle
		$$ CH^2(X) \otimes \Z_{\ell} \to H^4(X,\Z_{\ell}(2))$$ est surjective.	
	\end{thm}  
	
	Lorsque le groupe de N\'eron-Severi g\'eom\'etrique de la surface $S$ dans le th\'eor\`eme \ref{mainthm1} est sans torsion, on retrouve le r\'esultat de Pirutka. 
	
	\begin{thm}\label{mainthm2}
		Supposons que
		la surface
		$S$ est g\'eom\'etriquement $CH_{0}$-triviale et que l'on a  
		\begin{equation}\label{cond2}{\rm Hom}_{G}({\rm NS}(\cl{S})\{\ell\}, J(C)(\cl{\F}))=0.
		\end{equation} 
		Alors la fl\`eche naturelle $H^3_{\on{nr}}({\F}(X), \Q_{\ell}/\Z_{\ell}(2))\to H^3_{\on{nr}}({\cl{\F}}(X),\Q_{\ell}/\Z_{\ell}(2))$ est injective.
	\end{thm} 
	
	La preuve du th\'eor\`eme \ref{mainthm2} est le c{\oe}ur technique de notre travail.
	La condition restrictive (\ref{cond2}) est requise pour notre d\'emonstration. Nous ne savons pas si elle est n\'ecessaire. Il faut souligner que sous cette condition la surjectivit\'e de (\ref{tate-int1}) 
	et (\ref{tate-int2}) est facile \`a d\'emontrer  (voir la remarque \ref{tateclassique} (i)), mais ceci n'est pas le cas pour (\ref{tate-int3}).
	C'est une autre manifestation du fait que la version (\ref{tate-int3}) de la conjecture enti\`ere de Tate est beaucoup plus d\'elicate  que les versions (\ref{tate-int1}) et (\ref{tate-int2}). 
	
	\begin{thm}\label{mainthm3}
		Supposons que la surface $S$ est g\'eom\'etriquement $CH_{0}$-triviale, que la condition restrictive (\ref{cond2}) est satisfaite et que l'application cycle
		$$CH^2(\cl{X})\otimes \Z_{\ell} \to H^4(\cl{X},\Z_{\ell}(2))$$ est surjective, alors $H^3_{\on{nr}}(\F(X),\Q_{\ell}/\Z_{\ell}(2))=0$ et l'application cycle
		$$CH^2(X)\otimes \Z_{\ell} \to H^4(X,\Z_{\ell}(2))$$
		est surjective.
	\end{thm}  
	Le th\'eor\`eme \ref{mainthm3} ne d\'epend pas de la conjecture de Tate pour les surfaces sur un corps fini. 
	Si $X$ est le produit d'une courbe $C$ et d'une surface $S$ g\'eom\'etriquement $CH_{0}$-triviale, 
	$H^4(\cl{X},\Z_{\ell}(2))^{(1)}=H^4(\cl{X},\Z_{\ell}(2))$. Si la conjecture de Tate vaut pour toutes les surfaces sur un corps fini, le th\'eor\`eme de Schoen implique  alors que l'application $CH^2(\cl{X})\otimes\Z_{\ell}\to H^4(\cl{X},\Z_{\ell}(2))$ est surjective, c'est-\`a-dire, $X$ satisfait l'hypoth\`ese suppl\'ementaire du th\'eor\`eme \ref{mainthm3}. On obtient ainsi - pour nos cas particuliers, voir le th\'eor\`eme \ref{presquemainthmTate} - un \'enonc\'e de la forme du th\'eor\`eme de Schoen, mais avec (\ref{tate-int2}) remplac\'ee par (\ref{tate-int3}): 
	
	\begin{thm}\label{mainthm4}
		Sous la conjecture de Tate pour les surfaces, si $X=C\times_{\F}S$ 
		o\`u
		$C$ est une courbe, $S$ est une surface $CH_0$-triviale et la condition restrictive (\ref{cond2}) est satisfaite, alors l'application cycle $CH^2(X)\otimes \Z_{\ell} \to H^4(X,\Z_{\ell}(2))$ est surjective.
	\end{thm}  
	
	Comme exemple d'application des th\'eor\`emes \ref{mainthm2} et \ref{mainthm4}, on peut prendre pour $S$ une surface d'Enriques  sur un corps fini  $\F$, et pour $C$ une courbe elliptique $E$ sur $\F$.
	Pour $\ell \neq 2$,  la condition (\ref{cond2}) est automatiquement satisfaite. Si $\ell=2$, la condition (\ref{cond2}) est que la courbe elliptique n'a pas de point de $2$-torsion
	non nul d\'efini sur $\F$, c'est-\`a-dire, que $E$ a un mod\`ele affine d'\'equation $y^2=f(x)$, o\`u $f(x)$ est un polynome irr\'eductible de degr\'e $3$. Donc cette condition est remplie ``une fois sur deux".
	
	Nous d\'ecrivons maintenant les principaux ingr\'edients de nos preuves. Des techniques de $K$-th\'eorie alg\'ebrique combin\'ees aux conjectures de Weil mon\-trent (\cite[Th\'eor\`eme 6.8]{colliot2013cycles}) que, pour toute vari\'et\'e $X$ projective, lisse et g\'eom\'e\-tri\-quement connexe sur un corps fini $\F$,  on a une suite exacte longue de groupes de torsion
	\begin{equation}\label{ctkahn}
		\hspace*{-0.5cm}0\to \on{Ker}[CH^2(X)\{\ell\} \to CH^2(\cl{X})\{\ell\}    ] \to H^1(\F,  H^3(\cl{X},\Z_{\ell}(2))_{\on{tors}}) \end{equation} \[\hspace*{3cm} \to \on{Ker}[H^3_{\on{nr}}(\F(X),\Q_{\ell}/\Z_{\ell}(2) \to H^3_{\on{nr}}(\cl{\F}(X),\Q_{\ell}/\Z_{\ell}(2)) ]  \] \[\hspace*{6.3cm} \to \on{Coker}[CH^2(X)\to CH^2(\cl{X})^G] \{\ell\}\to 0.\]

	Sous certaines hypoth\`eses sur $X$, on cherche \`a \'etablir la nullit\'e du  groupe  $$\on{Ker}[H^3_{\on{nr}}(\F(X),\Q_{\ell}/\Z_{\ell}(2)) \to H^3_{\on{nr}}(\cl{\F}(X),\Q_{\ell}/\Z_{\ell}(2))]$$ en analysant d'une part la nullit\'e de la fl\`eche  
	$$H^1(\F,  H^3(\cl{X},\Z_{\ell}(2))_{\on{tors}})\to H^3_{\on{nr}}(\F(X),\Q_{\ell}/\Z_{\ell}(2)) \subset H^3(\F(X),\Q_{\ell}/\Z_{\ell}(2))$$
	d'autre part la trivialit\'e du groupe $$\on{Coker}[CH^2(X)\to CH^2(\cl{X})^G]\{\ell\}$$	
	dans (\ref{ctkahn}).  
	
	\medskip
	La   nullit\'e de la fl\`eche mentionn\'ee  est un pur probl\`eme de cohomologie
	$\ell$-adique, qui se transcrit ainsi  :

	{\it  Pour  $X$ projective,  lisse, g\'eom\'etriquement connexe sur un corps fini,
		le  noyau de la fl\`eche  $H^3(X,\Q_{\ell}/\Z_{\ell}(2)) \to H^3(\overline{X},\Q_{\ell}/\Z_{\ell}(2))$
		a-t-il une image nulle dans le groupe $ H^3(\F(X), \Q_{\ell}/\Z_{\ell}(2))$ ?}
	
	\`A  notre connaissance,  c'est un probl\`eme ouvert.  
	C'est connu pour $X$ de dimension au plus 2.  
	C'est  un des premiers r\'esultats de la th\'eorie du corps de classes sup\'erieur
	\cite{colliot1983torsion, katoclf}.
	Au \S \ref{noyauCH2}, nous montrons (Corollaire \ref{courbesurfaceqcque}) que c'est aussi le cas
	pour tout produit d'une surface et d'un nombre quelconque de courbes:
	\begin{thm}\label{courbesurface} 
		Soit $\F$ un corps fini et $\ell$ un premier distinct de la caract\'eristique de $\F$.
		Soit $S/\F$ une surface projective lisse, g\'eom\'etriquement connexe.
		Soit  $Y/\F$ un produit de courbes projectives, lisses, g\'eom\'etriquement connexes.
		Soit  $X= Y \times S$.
		Alors le  noyau de la fl\`eche  $H^3(X,\Q_{\ell}/\Z_{\ell}(2)) \to H^3(\overline{X},\Q_{\ell}/\Z_{\ell}(2))$
		a une image nulle dans le groupe $ H^3(\F(X), \Q_{\ell}/\Z_{\ell}(2))$. 
\end{thm}
	 
	C'est une cons\'equence  d'un \'enonc\'e g\'en\'eral, le Th\'eor\`eme \ref{noyau}.
	
	\medskip

	L'\'etude du groupe  $\on{Coker}(CH^2(X)\to CH^2(\cl{X})^G)$ (groupe qui est de torsion)  pour un produit $X= C \times S$ 
	d'une courbe et d'une surface est plus d\'elicate.  
	Au \S \ref{conoyauCH2}, sous l'hypoth\`ese que  la surface $S$ est g\'eom\'etriquement $CH_0$-triviale, nous donnons
	une condition \'equi\-va\-lente \`a l'injectivit\'e de la restriction
	$$H^3_{\on{nr}}({\F}(X), \Q_{\ell}/\Z_{\ell}(2))\to H^3_{\on{nr}}({\cl{\F}}(X),\Q_{\ell}/\Z_{\ell}(2)) $$
	en termes du groupe de Chow des  z\'ero-cycles de degr\'e 0 sur  la surface $S\times_{\F}{\F(C)}$, \`a savoir: 
	\begin{equation}\label{cond36}
		\text{L'homomorphisme $A_{0}(S_{\F(C)})\{\ell\} \to A_{0}(S_{\cl{\F}(C)})\{\ell\}^G$ est surjectif.}
	\end{equation}
	Ce r\'esultat (Th\'eor\`eme \ref{presquemainthm}) est obtenu par un calcul de correspondances sur le produit $X= C \times S$. On l'utilise avec des r\'esultats sur les exposants de torsion des surfaces qui reposent sur \cite{colliot1985k} pour d\'emontrer le  th\'eor\`eme \ref{mainthm1}.
	
	Au \S \ref{conoyauCH2bis},  en utilisant des outils \'elabor\'es venant de la $K$-th\'eorie alg\'ebrique,
	nous montrons que la condition (\ref{cond36})
	est satisfaite sous l'hypoth\`ese 
	suppl\'e\-men\-taire
	(\ref{cond2}), car elle implique $ A_{0}(S_{\cl{\F}(C)})\{\ell\}^G=0$. 
	Ceci ach\`eve la d\'emonstration du th\'eor\`eme \ref{mainthm2}.
	La  preuve du th\'eor\`eme \ref{mainthm3}, 
	portant sur la conjecture de Tate enti\`ere pour les 1-cycles,
	est aussi \'etablie, \`a l'aide des r\'esultats rappel\'es au \S 
	\ref{cyclesTate}.
	Le \S
	\ref{conoyauCH2bis}
	se termine par deux remarques. La premi\`ere remarque fait le  parall\`ele avec le travail de Benoist et Ottem \cite{benoist2018failure}
	sur la conjecture de Hodge enti\`ere pour le produit d'une courbe et d'une surface d'Enriques complexes. La seconde remarque, d\'ej\`a mentionn\'ee ci-dessus, donne une preuve simple de la surjectivit\'e de (\ref{tate-int1}) pour $X$ comme dans 
	le
	th\'eor\`eme \ref{mainthm2}.
	Au \S \ref{cyclesTate} on donne des rappels sur l'application cycle en cohomologie \'etale
	$\ell$-adique 
	pour les vari\'et\'es projectives et lisses sur un corps fini,
	la question de la surjectivit\'e de cette application pour les $1$-cycles,
	et le lien de cette question, dans le cas de vari\'et\'es $X$ de dimension 3,
	avec la nullit\'e \'eventuelle du groupe $H^3_{\on{nr}}({\F}(X), \Q_{\ell}/\Z_{\ell}(2))$.
	Comme indiqu\'e ci-dessus, ces rappels sont utilis\'es dans la d\'emonstration du 
	th\'eor\`eme \ref{mainthm3}.

	\subsection*{Notations et rappels}
	Si $A$  est un groupe ab\'elien, $n\geq 1$ est un entier, et $\ell$  est un nombre premier,  on note $A[n]:=\set{a\in A: na=0}$, $A\set{\ell}$ le sous-groupe de torsion $\ell$-primaire de $A$, $A_{\on{tors}}$ le sous-groupe de torsion de $A$, et $A/\on{tors}:=A/A_{\on{tors}}$. Si $B$ est un autre groupe ab\'elien, on note $A\otimes B:=A\otimes_{\Z}B$ et $\on{Hom}(A,B):=\on{Hom}_{\Z}(A,B)$. 
	Si $p$ est un nombre premier, on note
	$A_{p}:=A\otimes \Z[1/p]$.

	Si $k$ est un corps, on note $k^*:=k\setminus \set{0}$ le groupe multiplicatif de $k$, et $\cl{k}$ une cl\^{o}ture s\'eparable de $k$. Si $k'/k$ est une extension galoisienne de corps, $G=\on{Gal}(k'/k)$ est le groupe de Galois, et $M$ est un $G$-module continu, on note   $H^i(G,M)$ le $i$-\`eme groupe de cohomologie galoisienne de $G$ \`a valeurs dans $M$, et on \'ecrit $M^G:=H^0(G,M)$ pour le sous-module des \'el\'ements fix\'es par $G$. Si $k'=\cl{k}$, on \'ecrit $H^i(k,M)$ pour $H^i(G,M)$. Si $X$ est un sch\'ema et $F$ un faisceau 
	pour la topologie \'etale sur $X$, on note $H^{i}(X,F)=H^{i}_{\text{\'et}}(X,F)$ les groupes
	de cohomologie \'etale.
	
	Si $k$ est un corps, $\ell$   un nombre premier diff\'erent de la caract\'eristique de $k$, et $n\geq 1$ un entier, on note  $\mu_{{\ell}^n}$ 
	le faisceau ab\'elien \'etale sur $X$ associ\'e au
	sch\'ema en groupes fini \'etale des racines $\ell^n$-i\`emes de l'unit\'e. Si $i>0$, on note $\mu_{{\ell}^n}^{\otimes i}:=\mu_{{\ell}^n}\otimes\dots\otimes\mu_{{\ell}^n}$ ($i$ fois), et si $i<0$ on note $\mu_{{\ell}^n}^{\otimes i}:=\on{Hom}(\mu_{{\ell}^n}^{\otimes (-i)},\Z/{\ell}^n)$, et $\mu_{{\ell}^n}^{\otimes 0}:= \Z/\ell^n$. Si $X$ est un $k$-sch\'ema, on note  $H^i(X,\mu_{{\ell}^n}^{\otimes j})$ les groupes
	de cohomologie \'etale associ\'es.  
	On  note	$H^i(X,\Q_{\ell}/\Z_{\ell}(j))$ la limite directe des 
	$H^i(X,\mu_{{\ell}^n}^{\otimes j})$ pour $n$ tendant vers l'infini. On note 
	$H^i(X,\Z_{\ell}(j))$ la limite projective des $H^i(X,\mu_{{\ell}^n}^{\otimes j})$ et 
	$H^i(X,\Q_{\ell}(j)):= H^i(X,\Z_{\ell}(j))\otimes_{\Z_{\ell}}{\Q}_{\ell}$. Si $k$ est s\'eparablement clos et $X$ est propre et lisse, on note   $b_i(X):=\on{dim}_{\Q_{\ell}}H^i(X,\Q_{\ell})$ les nombres de Betti de $X$.
	
	Si $k'/k$ est une extension de corps et $X$ est un $k$-sch\'ema, on note $X_{k'}:=X\times_kk'$. Si $k'=\cl{k}$, on \'ecrit $\cl{X}$ pour $X\times_k\cl{k}$.
	
	Une vari\'et\'e sur $k$ est un sch\'ema s\'epar\'e de type fini sur le corps $k$. Si $X$ est un sch\'ema, on note $H^i(X,\mc{O}_X)$ les groupes de cohomologie de Zariski \`a valeurs dans le faisceau structural $\mc{O}_X$, on note $\on{Pic}(X)$ le groupe de Picard $H^1_{\on{Zar}}(X,\G_{\on{m}}) \simeq H^1_{\text{\'et}}(X,\G_{\on{m}})$, et on note $\on{Br}(X)$ le groupe de Brauer cohomologique $H^2_{\text{\'et}}(X,\G_{\on{m}})$. Si $X$ est une vari\'et\'e connexe, propre et lisse sur un corps s\'eparablement clos, on note $\on{NS}(X)$   le groupe de N\'eron-Severi de $X$ (qui est un groupe de type fini), et on note $\rho(X)$ le rang de $\on{NS}(X)$. Si $X$ est une vari\'et\'e  propre, lisse,
	g\'eom\'etriquement connexe sur un corps $k$, on note   $q(X)$ la dimension du $k$-espace vectoriel $H^1(X,\mc{O}_X)$.
	
	Si $X$ est une vari\'et\'e sur un corps $k$, on note   $CH_i(X)$ le groupe des cycles de dimension $i$ modulo \'equivalence rationnelle. Si $X$ est lisse et irr\'eductible de dimension $d$, on note $CH^i(X):= CH_{d-i}(X)$. Si $X$ est propre sur $k$, on note  $A_0(X)$ le noyau de l'homomorphisme de degr\'e $CH_{0}(X)\to \Z$. Si $X$ est irr\'eductible, on note  $X^{(j)}$ l'ensemble des points de codimension $j$ de $X$.
	
	Si $\ell$ est un nombre premier inversible dans $\mc{O}_X$, on note $\mc{H}^i(X,\mu_{{\ell}^n}^{\otimes j})$ le faisceau Zariski sur $X$ associ\'e au pr\'efaisceau $U\mapsto H^i(U,\mu_{{\ell}^n}^{\otimes j})$. Si $X$ est un sch\'ema, on note  $\mc{K}_i$ le faisceau Zariski sur $X$ associ\'e au pr\'efaisceau $U\mapsto K_i(H^0(U,\mc{O}_X))$, $K_i$ \'etant le $i$-\`eme groupe de $K$-th\'eorie de Quillen.

	Terminons par un rappel sur les vari\'et\'es g\'eom\'etriquement $CH_{0}$-trivales.
	Soient $k$ un corps d'exposant caract\'eristique $p$ et $X$ une $k$-vari\'et\'e projective, 
	lisse, g\'eom\'e\-triquement connexe.  Supposons que $X$ est {\it g\'eom\'etriquement
		$CH_{0}$-trivale}, c'est-\`a-dire que pour tout corps alg\'ebriquement clos $\Omega$
	contenant $k$, l'application  degr\'e  $CH_{0}(X_{\Omega})\otimes\Q  \to \Q$ est un isomorphisme
	(pour une propri\'et\'e \'equivalente, voir le lemme \ref{albanese}). 
	Si cette propri\'et\'e est satisfaite, alors,
	pour tout corps $F$ contenant $k$, le noyau $A_{0}(X_{F})$
	de la fl\`eche degr\'e $CH_{0}(X_{F})\to \Z$ est un groupe de torsion.
	Supposons de plus $k$ s\'eparablement clos. Un argument de correspondances
	(voir \cite[\S 1]{auelctparimala}) montre qu'alors  les propri\'et\'es suivantes
	sont satisfaites :

	(i) Pour tout $\ell$ premier, $\ell\neq p$, on a   $b_{1}(X)={\rm dim}(H^1(X,\Q_{\ell}))=0$.
	La vari\'et\'e de Picard r\'eduite $\on{Pic}^0_{X/k,{\rm red}}$ est triviale,
	et la vari\'et\'e d'Albanese, dont la dimension est $b_{1}$, est triviale.
	Le groupe de Picard $\on{Pic}(X)$
	de $X$ co\"{\i}ncide avec le groupe de N\'eron-Severi $\on{NS}(X)$  de $X$,
	qui est un groupe de type fini.
	[On n'a pas n\'ecessairement $q={\rm dim}(H^1(X,\mc{O}_X))=0$, comme le montre l'exemple
	de certaines surfaces d'Enriques supersinguli\`eres en caract\'eristique 2,
	qui sont (ins\'eparablement) unirationnelles.]
	
	(ii) Pour tout premier $\ell \neq p$, la fl\`eche naturelle $\on{NS}(X)\otimes \Z_{\ell} \to H^2(X,\Z_{\ell}(1))$
	est un isomorphisme. En particulier 
	la dimension $\rho$ de l'espace vectoriel $\on{NS}(X)\otimes \Q_{\ell}$
	est \'egal au rang $b_{2}$ du $\Q_{\ell}$-vectoriel $H^2(X,\Q_{\ell})$.
	
	(iii) Si $\on{dim}(X)=2$,  la forme d'intersection sur $\on{Num}(X)=\on{NS}(X)/\on{NS}(X)_{\rm tors}$
	a son discriminant
	de la forme $\pm p^r$ pour un  entier $r\geq 0$.
	
	Au moins en caract\'eristique z\'ero, 
	Spencer Bloch a conjectur\'e que  toute surface  connexe $X$ projective et lisse sur le corps des complexes avec $b_{1}=0$ et $b_{2}-\rho=0$, 
	est g\'eom\'etriquement $CH_{0}$-trivale. 
	Cela a \'et\'e \'etabli pour certaines surfaces
	\cite{bkl}, parmi lesquelles les surfaces d'Enriques. 
	
	En caract\'eristique $p$ quelconque, y compris $p=2$, 
	les surfaces d'Enriques sont g\'eom\'etriquement $CH_{0}$-trivales. 
	En  caract\'eristique diff\'erente de 2,
	ceci se voit
	en suivant la d\'emonstration de \cite{bkl},
	soit en utilisant le fait que ces surfaces se rel\`event en des surfaces d'Enriques
	en caract\'eristique z\'ero \cite[Proof of Thm. 1.1]{wlang}, et en utilisant
	la fl\`eche de sp\'ecialisation de Fulton sur les groupes de Chow.
	En caract\'eristique 2, les r\'esultats de rel\`evement et d'unirationalit\'e
	\cite{wlang,blass} donnent le r\'esultat. 
	
	La conjecture de Bloch a \'et\'e \'etablie pour des surfaces de type g\'en\'eral, par exemple pour la surface de Godeaux
	quotient de la surface de Fermat de degr\'e $5$ par $\Z/5\Z$, si $p\neq 5$ (voir \cite[Theorem 1, Remark p. 210]{inose1979rational}).

	\section{Le noyau de $CH^2(X)\to CH^2(\cl{X})$}\label{noyauCH2}
	
	\subsection{Rappels}\label{rappels}
	
	Rappelons d'abord des arguments remontant \`a des travaux de Spencer Bloch, et qu'on peut trouver
	d\'etaill\'es dans \cite{colliot1983torsion} et \cite{colliot1993cycles}.
	
	Pour $X/k$ une vari\'et\'e lisse sur un corps, et $n>0$ entier premier \`a
	la caract\'eristique de $k$, on dispose de suites exactes
	\[ 0 \to H^1_{\on{Zar}}(X,{\mc{K}}_{2})/n \to H^1_{\on{Zar}}(X,{\mc{K}}_{2}/n) \to CH^2(X)[n] \to 0\]
	et si $X$ est de plus int\`egre
	\[ 0 \to H^1_{\on{Zar}}(X,{\mc{H}}^2(\mu_{n}^{\otimes 2})) \to H^3(X,\mu_{n}^{\otimes 2})
	\to H^3(k(X), \mu_{n}^{\otimes 2});\] voir \cite[(3.11), (3.10)]{colliot1993cycles} 
	(la premi\`ere utilise la conjecture de Gersten pour la $K$-th\'eorie, \'etablie par Quillen, la deuxi\`eme utilise la conjecture de Gersten en cohomologie \'etale, \'etablie par Bloch et Ogus).
	Ces deux r\'esultats et le th\'eor\`eme de Merkurjev-Suslin donnent un isomorphisme de  faisceaux
	${\mc{K}}_{2}/n \xrightarrow{\sim} {\mc{H}}^2(\mu_{n}^{\otimes 2})$.
	
	Soient d\'esormais $k=\F$ un corps fini et $X$ une vari\'et\'e  projective, lisse, g\'eom\'e\-tri\-que\-ment connexe
	sur $\F$.  Soit $\ell$ un nombre premier distinct de la
	caract\'eristique de $\F$. En passant \`a la limite directe sur les puissances de $\ell$ on obtient  des suites exactes
	$$ 0 \to H^1_{\on{Zar}}(X,{\mc{K}}_{2})\otimes \Q_{\ell}/\Z_{\ell} \to H^1_{\on{Zar}}(X,{\mc{H}}^2(\Q_{\ell}/\Z_{\ell}(2))) \to
	CH^2(X)\{\ell\}  \to 0$$
	et 
	$$0 \to H^1_{\on{Zar}}(X,{\mc{H}}^2(\Q_{\ell}/\Z_{\ell}(2))) \to H^3(X, \Q_{\ell}/\Z_{\ell}(2)) \to H^3(\F(X), \Q_{\ell}/\Z_{\ell}(2)).$$
	Il r\'esulte des conjectures de Weil prouv\'ees par Deligne que le groupe 
	$H^3(X, \Q_{\ell}/\Z_{\ell}(2))$ est fini; voir \cite[Th\'eor\`eme 2, p.780]{colliot1983torsion}. Ainsi la fl\`eche compos\'ee
	\[H^1_{\on{Zar}}(X,{\mc{K}}_{2})\otimes \Q_{\ell}/\Z_{\ell} \to H^1_{\on{Zar}}(X,{\mc{H}}^2(\Q_{\ell}/\Z_{\ell}(2))) \hookrightarrow 
	H^3(X, \Q_{\ell}/\Z_{\ell}(2))\]
	est nulle. On obtient en fin de compte la suite exacte
	\[ 0 \to CH^2(X)\{\ell\}  \to
	H^3(X, \Q_{\ell}/\Z_{\ell}(2)) \to H^3(\F(X), \Q_{\ell}/\Z_{\ell}(2)).\]
	
	Par passage  \`a la limite  sur les extensions de $\F$ on obtient 
	une suite exacte analogue
	pour $\cl{X}/\cl{\F}$. On a donc un diagramme commutatif de suites exactes:
	\[
	\begin{tikzcd}
		0 \arrow[r] & CH^2(X)\{\ell\}     \arrow[r] \arrow[d]  & H^3(X, \Q_{\ell}/\Z_{\ell}(2)) \arrow[r] \arrow[d] & H^3(\F(X), \Q_{\ell}/\Z_{\ell}(2)) \arrow[d]\\
		0 \arrow[r] & CH^2(\cl{X})\{\ell\} \arrow[r] & H^3(\cl{X}, \Q_{\ell}/\Z_{\ell}(2)) \arrow[r] & H^3(\cl{\F}(X), \Q_{\ell}/\Z_{\ell}(2)).  
	\end{tikzcd}
	\]		
	On en d\'eduit une suite exacte
	\begin{align}
		\label{equiv1}
		0 \to \on{Ker}[CH^2(X)\{\ell\}  \to CH^2(\cl{X})\{\ell\}]  \to  \hspace*{5cm}   \\
		\hspace*{1cm}  \on{Ker} [H^3(X,\Q_{\ell}/\Z_{\ell}(2)) \to  H^3(\cl{X},\Q_{\ell}/\Z_{\ell}(2)) ]
		\to  H^3(\F(X), \Q_{\ell}/\Z_{\ell}(2)),\nonumber \end{align}
	qui est
	le d\'ebut de la suite exacte  (\ref{ctkahn}).
	
	Soit $n\geq 1$ un entier. Si $U\subset X$ est un ouvert, on a la suite spectrale de Hochschild-Serre en cohomologie \'etale
	\[E^{p,q}_2:= H^p(\F, H^q(\cl{U}, \mu_{{\ell}^n}^{\otimes 2}))    \Longrightarrow 
	H^{p+q}(U,\mu_{{\ell}^n}^{\otimes 2}).\]
	Pour chaque ouvert $U\subset X$,  on a donc une suite exacte
	\begin{equation}\label{equiv2}0 \to H^1(\F, H^2({\overline U},\mu_{{\ell}^n}^{\otimes 2})) \to    H^3(U,\mu_{{\ell}^n}^{\otimes 2}) \to H^3({\overline U},\mu_{{\ell}^n}^{\otimes 2})\end{equation} 
	et aussi une suite exacte
	\begin{equation}\label{equiv3}0 \to H^1(\F, H^2(\cl{\F}(X), \mu_{{\ell}^n}^{\otimes 2})) \to H^3(\F(X), \mu_{{\ell}^n}^{\otimes 2}) \to H^3(\cl{\F}(X), \mu_{{\ell}^n}^{\otimes 2}).\end{equation}
	
	Par limite directe sur les coefficients  $\mu_{{\ell}^n}^{\otimes 2}$ pour $n$ tendant vers l'infini, on a 
	donc les suites exactes
	\begin{equation}\label{equiv2'}0 \to H^1(\F, H^2({\overline U},\Q_{\ell}/\Z_{\ell}(2))) \to    H^3(U,\Q_{\ell}/\Z_{\ell}(2)) \to H^3({\overline U},\Q_{\ell}/\Z_{\ell}(2))\end{equation} 
	et
	\begin{equation}\label{equiv3'}0 \to H^1(\F, H^2(\cl{\F}(X), \Q_{\ell}/\Z_{\ell}(2))) \to H^3(\F(X), \Q_{\ell}/\Z_{\ell}(2)) \to H^3(\cl{\F}(X), \Q_{\ell}/\Z_{\ell}(2)).\end{equation}

	On a donc une injection 
	$$ \theta_{\ell} : \on{Ker}[CH^2(X)\{\ell\}  \to CH^2(\cl{X})\{\ell\}] \hookrightarrow  H^1(\F, H^2(\overline{X},\Q_{\ell}/\Z_{\ell}(2)))$$
	Ces deux groupes sont finis, car
	le th\'eor\`eme de Deligne sur les valeurs propres de Frobenius implique que les groupes
	$H^1(\F, H^2(\overline{X},\Q_{\ell}/\Z_{\ell}(2)))$ et $H^3(X,\Q_{\ell}/\Z_{\ell}(2))$ sont finis; voir \cite[Th\'eor\`eme 2, p.780]{colliot1983torsion}. 
	Ce th\'eor\`eme implique aussi que l'on a
	un isomorphisme de groupes finis
	$$\rho_{\ell} : H^1(\F, H^2(\overline{X},\Q_{\ell}/\Z_{\ell}(2))) \xrightarrow{\sim} H^1(\F, H^3(\overline{X},\Z_{\ell}(2))_{\on{tors}}).$$
	En composant avec l'application $\theta_{\ell}$, on obtient une injection de groupes finis
	$$ \phi_{\ell} : \on{Ker}[CH^2(X)\{\ell\}  \to CH^2(\cl{X})\{\ell\}] \hookrightarrow H^1(\F, H^3(\overline{X},\Z_{\ell}(2))_{\on{tors}}).$$
	
	\subsection{La fl\`eche $\phi_{\ell}$ est-elle un isomorphisme ?}
	
	\begin{prop}\label{surj-char} Soit $X$ une vari\'et\'e projective, lisse, g\'eom\'etriquement connexe sur un corps fini $\F$. Soit $\ell$ un premier distinct de la caract\'eristique de $\F$.
		Les \'enonc\'es suivants sont \'equivalents:
		
		\begin{enumerate}[label=(\roman*)]
			\item L'application $\theta_{\ell}$ est un isomorphisme.
			\item L'application $\phi_{\ell}$ est un isomorphisme.
			\item Le compos\'e de l'injection de groupes finis
			\[H^1(\F, H^2(\overline{X},\Q_{\ell}/\Z_{\ell}(2))) \to H^3(X,\Q_{\ell}/\Z_{\ell}(2))\]
			et de l'application
			\[H^3(X,\Q_{\ell}/\Z_{\ell}(2)) \to H^3(\F(X), \Q_{\ell}/\Z_{\ell}(2))\]
			est nul.
			\item Le  noyau de la fl\`eche 
			$H^3(X,\Q_{\ell}/\Z_{\ell}(2)) \to H^3(\overline{X},\Q_{\ell}/\Z_{\ell}(2))$
			a une image nulle dans le groupe $ H^3(\F(X), \Q_{\ell}/\Z_{\ell}(2)).$
			\item La  fl\`eche naturelle
			$H^1(\F, H^2(\overline{X},\Q_{\ell}/\Z_{\ell}(2))) \to H^1(\F, H^2(\cl{\F}(X),\Q_{\ell}/\Z_{\ell}(2)))$
			est nulle.
			\item Pour tout entier $n\geq 1$, le noyau de la fl\`eche
			$H^3(X,\mu_{\ell^n}^{\otimes 2}) \to H^3(\overline{X}, \mu_{\ell^n}^{\otimes 2}) $
			a une image nulle dans 
			$H^3(\F(X),\mu_{\ell^n}^{\otimes 2}).$
			\item La propri\'et\'e pr\'ec\'edente vaut pour tout entier $n$ suffisamment grand.
			
			\item Pour tout entier $n\geq 1$, la fl\`eche naturelle
			\[H^1(\F, H^2(\overline{X}, \mu_{\ell^n}^{\otimes 2}))  
			\to H^1(\F, H^2(\cl{\F}(X), \mu_{\ell^n}^{\otimes 2})) \subset H^3(\F(X), \mu_{\ell^n}^{\otimes 2})    \]
			est nulle.
			\item La propri\'et\'e pr\'ec\'edente vaut pour tout entier $n$ suffisamment grand.
		\end{enumerate} 
	\end{prop} 
	
	\begin{proof} 
		L'\'equivalence des propri\'et\'es (i) \`a (v) est formelle
		\`a partir des suites exactes mentionn\'ees ci-dessus, qui montrent aussi
		que les propri\'et\'es (vi) et (vii) sont \'equiva\-lentes.
		La propri\'et\'e (vii) implique la propri\'et\'e (v) par limite inductive.
		Que la propri\'et\'e (iv) implique la propri\'et\'e (vi)  est une cons\'equence
		du th\'eor\`eme de Merkurjev-Suslin \cite{merkurjev1982k-cohomology}, qui implique que pour tout corps $k$
		de caract\'eristique diff\'erente de $\ell$, les fl\`eches naturelles
		$H^3(k, \mu_{\ell^{n}}^{\otimes 2}) \to H^3(k, \Q_{\ell}/\Z_{\ell}(2))$
		sont injectives.
	\end{proof}
	\begin{rmk}\label{remarquedim2}
		(i) L'application de groupes finis \[t:H^1(\F, H^3(\overline{X},\Z_{\ell}(2))_{\on{tors}}) \to H^1(\F, H^3(\overline{X},\Z_{\ell}(2)))\]
		est injective mais pas a priori surjective.
		
		Soit $X=C\times_{\F}S$,  o\`u $C$ et $S$ sont vari\'et\'es irr\'eductibles, projectives et lisses de dimension $1$ et $2$, respectivement. Comme la dimension cohomologique de $\F$ est \'egale \`a $1$, 
		on a \[\on{Coker}t\simeq H^1(\F, H^3(\overline{X},\Z_{\ell}(2))/\on{tors}).\]
		Montrons que ce groupe n'est pas nul en g\'en\'eral. Par la formule de K\"unneth $\ell$-adique, $(H^2(\cl{S},\Z_{\ell}(1))/\on{tors})\otimes H^1(\cl{C},\Z_{\ell}(1))$ est un facteur direct galoisien de $H^3(\overline{X},\Z_{\ell}(2))/\on{tors}$, et donc 
		\[H^1(\F,(H^2(\cl{S},\Z_{\ell}(1))/\on{tors})\otimes H^1(\cl{C},\Z_{\ell}(1)))\subset \on{Coker}t.\]
		(On a \'egalit\'e si $b_1(\cl{S})=0$.) Supposons que $b_2(\cl{S})=\rho(\cl{S})$, que $G$ agit trivialement sur  $\on{NS}(\cl{S})$ et que $J(\F)\set{\ell}\neq 0$, $J$ \'etant la jacobienne de $C$. Pour d\'emontrer 
		$\on{Coker}t\neq 0$, il suffit alors de v\'erifier que 
		l'on a
		$H^1(\F, H^1(\cl{C},\Z_{\ell}(1)))\neq 0$. La suite exacte courte
		\[0\to H^1(\F,H^1(\cl{C},\Z_{\ell}(1)))\to H^2(C,\Z_{\ell}(1))\to H^2(\cl{C},\Z_{\ell}(1))^G\to 0\]
		et le fait que l'on ait  $\on{Br}(C)=0$ entra\^{\i}nent
		\[H^1(\F, H^1(\cl{C},\Z_{\ell}(1)))\simeq J(\F)\set{\ell}\neq 0.\]
		
		(ii) On observera que les \'enonc\'es  (iii) \`a (vi) 
		de la proposition \ref{surj-char} 
		se formulent  purement en termes de la cohomologie
		\'etale des vari\'et\'es sur les corps finis, ils  ne font pas intervenir la $K$-th\'eorie alg\'ebrique.
		Mais pour passer des \'enonc\'es au niveau $\Q_{\ell}/\Z_{\ell}(2)$
		\`a tout niveau fini $\mu_{\ell^{n}}^{\otimes 2}$, il faut utiliser le th\'eor\`eme de Merkurjev-Suslin.
		
		(iii) Les \'enonc\'es de la proposition \ref{surj-char}  sont trivialement vrais pour $X$ de dimension~1. 
		Pour  $X$ de dimension $2$, \cite[Remarque 2, p. 790]{colliot1983torsion} implique que l'application de 
		restriction
		\[H^3(X,\Q_{\ell}/\Z_{\ell}(2)) \to H^3(\F(X), \Q_{\ell}/\Z_{\ell}(2))\]
		est nulle. Les
		\'enonc\'es (iii) et (iv)  de la proposition \ref{surj-char} 
		s'en suivent, et donc aussi les autres.

		\begin{comment} Comme $\rho_{\ell}$ est un isomorphisme, l'\'equivalence de (i) et (ii) est \'evidente.
		Des suites exactes (\ref{equiv1}) et (\ref{equiv2}) on d\'eduit que   $\theta_{\ell}$ est un isomorphisme
		si et seulement si la fl\'eche compos\'ee
		\[H^1(\F, H^2(\cl{X},\Q_{\ell}/\Z_{\ell}(2))) \to H^3(X,\Q_{\ell}/\Z_{\ell}(2)) \to H^3(\F(X), \Q_{\ell}/\Z_{\ell}(2))\]
		est nulle.\end{comment}
	\end{rmk}

	Le lemme suivant est bien connu.
	\begin{lemma}\label{endomiyata} 
		Soit $k$ un corps parfait dont le groupe de Galois absolu est procyclique. Soit $M$ un $k$-groupe de type multiplicatif fini,
		$\hat{M}$ son groupe des caract\`eres. Supposons que la torsion de $\hat{M}$
		est d'ordre premier \`a la caract\'eristique de $k$. 
		Alors il existe une suite exacte de $k$-groupes alg\'ebriques commutatifs lisses
		\[1 \to M \to P \to Q \to 1\]
		avec $P$ un $k$-tore quasitrivial et $Q$ un $k$-tore facteur direct
		d'un $k$-tore quasitrivial.
	\end{lemma}
	\begin{proof} Rappelons qu'un $k$-tore $T$ est appel\'e quasitrivial si son groupe de caract\`eres
		$\hat{T}$ est un module galoisien de permutation
		et qu'il est appel\'e coflasque si $H^1(K,\hat{T})=0$ pour toute extension de corps   $K/k$. D'apr\`es  Endo et Miyata \cite[Lemma 1.1]{endo1975classification}, pour  tout   $k$-groupe de type multiplicatif $M$
		il existe une suite exacte de modules galoisiens $0 \to \hat{Q} \to \hat{P} \to \hat{M} \to 0$ avec $\hat{P}$ de permutation et $\hat{Q}$ coflasque.
		Sous l'hypoth\`ese que le groupe de Galois absolu est procyclique, tout module galoisien sans torsion de type fini qui est coflasque est un facteur
		direct d'un module de permutation (Endo et Miyata, \cite[Theorem 1.5]{endo1975classification}). Par dualit\'e de Cartier, on conclut.
	\end{proof}

	\begin{prop}\label{mult} 
		Soit $X$ une vari\'et\'e projective et lisse, g\'eom\'etriquement connexe
		sur un corps fini $\F$.  Soient $\F(X)$ son corps des fonctions rationnelles, et \[\on{Br}_{1}(X):= {\rm Ker}[\on{Br}(X) \to \on{Br}(\cl{X})].\]
		Supposons que le module galoisien d\'efini par le groupe de N\'eron-Severi g\'eom\'etrique $\on{NS}(\cl{X})$  de $X$
		est facteur direct d'un module de permutation (en particulier, $\on{NS}(\cl{X})$ est sans-torsion). Alors :
		
		(a) On a $\on{Br}_{1}(X_{\F'})=0$ pour toute extension finie $\F'$ de $\F$.
		
		(b) Pour tout $\F$-tore $R$  facteur direct d'un $\F$-tore quasitrivial, 
		on a 
		$${\rm Ker} [H^2(X,R) \to H^2(\cl{X},R)] = 0.$$
		
		(c) Pour tout   $\F$-groupe de type multiplicatif fini $M$
		d'ordre premier \`a la caract\'eristique de $\F$, 
		l'image de ${\rm Ker} [H^2(X,M) \to H^2(\cl{X},M)]$ dans
		$H^2(\F(X),M)$ est nulle.\footnote{M\^{e}me si l'action du groupe de Galois absolu sur $\on{NS}(\cl{X})$
			est triviale,
			la restriction $\on{NS}(\cl{X})$ sans torsion
			est n\'ecessaire.
			Soit $X/\F$ une surface d'Enriques sur un corps fini $\F$
			de caract\'eristique impaire. 
			Supposons $\on{Pic}(X)=\on{Pic}(\cl{X})$. Pour $M=\mu_{2}$, l'\'enonc\'e (c)
			est en d\'efaut.}

		(d) Pour tout   $\F$-groupe de type multiplicatif fini $M$
		d'ordre premier \`a la caract\'eristique de $\F$, 
		l'image de $H^1(\F, H^1(\cl{X},M))$ dans $H^1(\F, H^1(\cl{\F}(X),M))$
		est nulle.
	\end{prop}
	
	\begin{proof} 
		Pour une vari\'et\'e projective   $X$ g\'eom\'etriquement int\`egre
		sur un corps $k$ et $\cl{k}$ une cl\^{o}ture s\'eparable de $k$, et $\cl{X}=X \times_{k} \cl{k}$,
		on a la suite exacte
		$$ \on{Br}(k) \to \on{Br}_1(X) \to H^1(k, \on{Pic}(\cl{X})).$$
		Si $X$ de plus est lisse, on a une suite exacte de modules galoisiens
		$$ 0 \to J_{X} (\cl{k}) \to \on{Pic}(\cl{X}) \to \on{NS}(\cl{X}) \to 0,$$
		o\`u $J_{X}= \on{Pic}_{X/k}^0$ est la vari\'et\'e de Picard, qui est
		une vari\'et\'e ab\'elienne, et $\on{NS}(\cl{X})$ est le groupe de N\'eron-Severi g\'eom\'etrique.
		Pour $k=\F$ un corps fini, 
		on a $H^1(\F,J_X)=0$ (Lang) et $\on{Br}(\F)=0$.
		Sous l'hypoth\`ese faite sur le groupe de N\'eron-Severi,
		on a $H^1(\F, \on{NS}(\cl{X}))=0$. On a donc  $\on{Br}_{1}(X)=0$. Comme les hypoth\`eses
		sont inchang\'ees, ceci vaut encore sur toute extension finie 
		de $\F$. 
		Ceci donne (a) et (b). 
		
		Soit
		\[1 \to M \to P \to Q \to 1\]
		une suite exacte donn\'ee par le lemme \ref{endomiyata}. 
		Comme $Q$ est facteur direct d'un tore quasitrivial,
		le lemme de Shapiro et le th\'eor\`eme 90 de Hilbert
		donnent $H^1(\F(X), Q)=0$. On a donc une injection 
		$H^2(\F(X),M) \hookrightarrow H^2(\F(X),P)$.
		
		On a   un diagramme commutatif
		\[
		\begin{tikzcd}
			H^2(X,M) \arrow[r] \arrow[d] &  H^2(X,P) \arrow[d] \\
			H^2(\F(X),M) \arrow[r] & H^2(\F(X),P).
		\end{tikzcd}	 
		\]
		D'apr\`es (b), le noyau de  ${\rm Ker}[H^2(X,M) \to H^2(\cl{X},M)]$
		a une image nulle dans $H^2(X,P)$.
		Il r\'esulte alors du diagramme que l'image de 
		${\rm Ker}[H^2(X,M) \to H^2(\cl{X},M)]$
		dans $H^2(\F(X),M)$ est nulle, soit (c).  
		L'\'enonc\'e (d) r\'esulte alors de la suite des termes de bas degr\'e 
		des suites spectrales de Hochschild-Serre.
	\end{proof}

	\begin{lemma}\label{torsionH2H3} 
		Soit $X$ une vari\'et\'e projective, lisse, connexe sur un corps $k$  s\'eparablement clos.
		Soit $\ell$ premier distinct de la caract\'eristique de $k$.
		
		(i) Si le groupe $\on{NS}(X)$ est sans $\ell$-torsion, alors pour tout entier $n>0$
		le groupe $H^1(X,\Z/{\ell}^n)$ est un $\Z/\ell^n$-module libre.
		
		(ii) Si de plus  le groupe $H^3(X, \Z_{\ell})$ est sans torsion, alors pour tout entier $n>0$
		le groupe $H^2(X,\Z/{\ell}^n)$ est un $\Z/\ell^n$-module libre.
	\end{lemma}
	\begin{proof}   Ceci r\'esulte d'arguments de passage \`a la limite
		dans les suites de Kummer  \cite[\S 8]{grothendieck1968brauer3}.
		La suite de Kummer pour la cohomologie \'etale donne
		un isomorphisme $H^1(X,\mu_{\ell^n} ) \simeq {\rm Pic}(X)[\ell^n]$. 
		Le groupe ${\rm Pic}(X)$ est une extension du groupe $\on{NS}(X)$ par
		le groupe des $k$-points de la vari\'et\'e de Picard $J$ de $X$.
		Sous l'hypoth\`ese de (i), on a donc $H^1(X,\mu_{\ell^n} ) \simeq J[\ell^n]$,
		qui est libre sur $\Z/\ell^n$.
		On sait que l'on a $\on{NS}(X)\{\ell\} \simeq H^2(X, \Z_{\ell}(1)) \{\ell\}$.
		Sous l'hypoth\`ese de (i), le $\Z_{\ell}$-module de type fini  $H^2(X,\Z_{l})$
		est donc un $\Z_{\ell}$-module libre.
		Sous l'hypoth\`ese  suppl\'ementaire de (ii), on a
		$H^2(X,\Z_{\ell})/ n \simeq H^2(X, \Z/{\ell^n})$.
	\end{proof}

	\begin{thm}\label{noyau} 
		Soient $\F$ un corps fini et $\ell$ un premier distinct de la caract\'eristique de $\F$. Soient $Y$ et $Z$ deux vari\'et\'es g\'eom\'etriquement connexes, projectives et lisses sur $\F$, de dimensions quelconques.  Supposons que le module galoisien d\'efini par le groupe de N\'eron-Severi g\'eom\'etrique $\on{NS}(\cl{Y})$  de $Y$ est facteur direct d'un module de permutation. Supposons de plus que le groupe $H^3(\cl{Y},\Z_{\ell})$ est sans torsion.
		Si les propri\'et\'es \'equivalentes de la proposition \ref{surj-char} valent pour $Y$ et pour $Z$, alors elles valent pour $X:=Y\times Z$. 
	\end{thm}
	
	\begin{proof} 
		Soit $n\geq 1$ un entier. Soient $U \subset Y$ et $V \subset Z$ des ouverts non vides.
		Soit $W: = U\times V$.
		On a un homomorphisme Galois-\'equivariant
		\begin{equation}
			H^2(\cl{U},\mu_{\ell^n}^{\otimes 2})\oplus [ H^1(\cl{U},\mu_{\ell^n})\otimes H^1(\cl{V},\mu_{\ell^n})]\oplus H^2(\cl{V},\mu_{\ell^n}^{\otimes 2})  \rightarrow H^2(\cl{W},\mu_{\ell^n}^{\otimes 2}).\end{equation}
		Les fl\`eches $H^2(\cl{U},\mu_{\ell^n}^{\otimes 2}) \to  H^2(\cl{W},\mu_{\ell^n}^{\otimes 2})$
		et  $H^2(\cl{V},\mu_{\ell^n}^{\otimes 2}) \to  H^2(\cl{W},\mu_{\ell^n}^{\otimes 2})$ sont induites par 
		les projections naturelles. La fl\`eche
		$ H^1(\cl{U},\mu_{\ell^n})\otimes H^1(\cl{V},\mu_{\ell^n}) \to H^2(\cl{W},\mu_{\ell^n}^{\otimes 2})$
		est induite par les projections naturelles et le cup-produit.

		D'apr\`es le lemme \ref{torsionH2H3}, les hypoth\`eses sur $\cl{Y}$ impliquent que les groupes de cohomologie
		$H^{i}(\cl{Y}, \Z/\ell^n)$ pour $i=0,1,2$ sont libres sur $\Z/\ell^n$.
		Le th\'eor\`eme  \cite[Chap. VI, Thm. 8.13]{milneetale} implique alors que
		l'homomorphisme 
		\begin{equation}\label{kuenferme}
			H^2(\cl{Y},\mu_{\ell^n}^{\otimes 2})\oplus [ H^1(\cl{Y},\mu_{\ell^n})\otimes H^1(\cl{Z},\mu_{\ell^n})]\oplus H^2(\cl{Z},\mu_{\ell^n}^{\otimes 2})  \rightarrow   H^2(\cl{X},\mu_{\ell^n}^{\otimes 2}).
		\end{equation}
		est un isomorphisme.

		La projection $X\to Z$ induit un diagramme commutatif d'homomorphismes
		\[
		\begin{tikzcd}
			H^1(\F,H^2(\cl{Z},\mu_{\ell^n}^{\otimes 2})) \arrow[d] \arrow[r]  &  H^3(Z,\mu_{\ell^n}^{\otimes 2}) \arrow[r] \arrow[d] & H^3(\F(Z),\mu_{\ell^n}^{\otimes 2}) \arrow[d] \\
			H^1(\F,H^2(\cl{X},\mu_{\ell^n}^{\otimes 2}))  \arrow[r]  &  H^3(X,\mu_{\ell^n}^{\otimes 2}) \arrow[r] & H^3(\F(X),\mu_{\ell^n}^{\otimes 2})   
		\end{tikzcd}	
		\]
		Par hypoth\`ese, la compos\'ee $H^1(\F,H^2(\cl{Z},\mu_{\ell^n}^{\otimes 2})) \to H^3(\F(Z),\mu_{\ell^n}^{\otimes 2}) $ est nulle.  Ainsi  la fl\`eche compos\'ee  $$H^1(\F,H^2(\cl{Z},\mu_{\ell^n}^{\otimes 2}))
		\to 		H^1(\F,H^2(\cl{X},\mu_{\ell^n}^{\otimes 2})) \to H^3(\F(X), \mu_{\ell^n}^{\otimes 2})$$
		est nulle. Le m\^eme argument vaut    avec 
		avec $Y$ en place de $Z$.		 		
		
		Pour \'etablir l'\'enonc\'e sur $X$, il reste \`a voir que, pour  tout $n\geq 1$,
		la fl\`eche  de modules galoisiens \[H^1(\cl{Y},\mu_{\ell^n})\otimes H^1(\cl{Z},\mu_{\ell^n})\to H^2(\cl{X},\mu_{\ell^n}^{\otimes 2})\to H^2(\cl{\F}(X),\mu_{\ell^n}^{\otimes 2})\] 
		induit une application nulle par application de $H^1(\F,-)$.

		Notons $M_{n}:= 	H^1(\cl{Z}, \mu_{\ell^n}) $.

		Consid\'erons le diagramme commutatif d'homomorphismes galoisiens
		\[
		\adjustbox{scale=0.8,center}{ 
			\begin{tikzcd}
				H^2(\cl{X},\mu_{\ell^n}^{\otimes 2})    \arrow[d] 
				&  \arrow[l] H^1(\cl{Y}, \mu_{\ell^n}) \otimes 	H^1(\cl{Z}, \mu_{\ell^n}) =  H^1(\cl{Y}, \mu_{\ell^n}) \otimes  M_{n}  \arrow[r] \arrow[d]  
				&   H^1(\cl{Y},  \mu_{\ell^n}    \otimes M_{n}) \arrow[d]  \\		
				H^2(\cl{\F}(X) ,\mu_{\ell^n}^{\otimes 2} )  
				&  \arrow[l]  H^1(\cl{\F}(Y), 	\mu_{\ell^n}) \otimes 	H^1(\cl{Z}, \mu_{\ell^n}) = H^1(\cl{\F}(Y), 	\mu_{\ell^n}) \otimes M_{n}  \arrow[r] 
				&  H^1(\cl{\F}(Y),  \mu_{\ell^n}\otimes M_{n}). 	
			\end{tikzcd}
		}
		\]	
		
		Dans ce diagramme, les fl\`eches horizontales de droite sont donn\'ees par les cup-produits 
		sur la cohomologie \'etale de $\cl{Y}$ et de $\cl{\F}(Y)$: 
		$$ H^1(\cl{Y}, \mu_{\ell^n}) \otimes   H^0(\cl{\F}, M_{n}) \to  H^1(\cl{Y},  \mu_{\ell^n}    \otimes M_{n}) $$
		et
		$$ H^1(\cl{\F}(Y), \mu_{\ell^n}) \otimes   H^0(\cl{\F}, M_{n}) \to  H^1(\cl{\F}(Y),  \mu_{\ell^n}    \otimes M_{n}).$$
		La fl\`eche verticale de droite est injective car $\cl{Y}$ est normale. 
		La fl\`eche horizontale inf\'erieure droite est
		Galois-\'equivariante pour l'action de 
		$\on{Gal}(\cl{\F}(Y)/\F(Y)) \simeq \on{Gal}(\cl{\F}/\F)$ (voir
		\cite[Prop. 3.4.10.4]{GSz}).  La fl\`eche horizontale sup\'erieure droite
		est donc $\on{Gal}(\cl{\F}/\F)$ \'equivariante.

		La fl\`eche inf\'erieure droite est par ailleurs un isomorphisme de modules galoisiens. Pour le voir, il suffit
		de consid\'erer les groupes ab\'eliens sous-jacents.
		Comme groupe ab\'elien, le module  $M_{n}=	H^1(\cl{Z}, \mu_{\ell^n})$ est une somme directe	
		$\oplus _{i }\Z/{\ell}^{r_{i}}$ avec $r_{i} \leq n$.
		Par la th\'eorie de Kummer, l'application 
		$$H^1(\cl{\F}(Y), 	\mu_{\ell^n}) \otimes 	M_{n}  = H^1(\cl{\F}(Y), 	\mu_{\ell^n}) \otimes 
		H^0( \cl{\F},	M_{n})
		\to  H^1(\cl{\F}(Y),  \mu_{\ell^n}\otimes M_{n})$$
		est donc un isomorphisme.

		L'hypoth\`ese que  $\on{NS}(\cl{Y})$  est un facteur direct d'un module de permutation
		et la proposition \ref{mult} (d) donnent que
		l'application 
		$$H^1(\F, H^1(\cl{Y},  \mu_{\ell^n}    \otimes M_{n})  ) \to H^1(\F, H^1(\cl{\F}(Y),  \mu_{\ell^n}\otimes M_{n}))$$
		est nulle.
		
		On en d\'eduit que l'application
		$$H^1(\F, H^1(\cl{Y}, \mu_{\ell^n}) \otimes 	H^1(\cl{Z}, \mu_{\ell^n}) ) \to H^1(\F,H^1(\cl{\F}(Y), 	\mu_{\ell^n}) \otimes 	H^1(\cl{Z}, \mu_{\ell^n}) )$$
		est nulle. Ceci implique alors que l'application  
		$$ H^1(\F, H^1(\cl{Y}, \mu_{\ell^n}) \otimes 	H^1(\cl{Z}, \mu_{\ell^n}) ) \to 
		H^1(\F, H^2(\cl{\F}(X) ,\mu_{\ell^n}^{\otimes 2} ))$$
		d\'eduite du diagramme commutatif ci-dessus est nulle, ce qu'il fallait d\'emontrer.
	\end{proof}

	\begin{rmk}\label{remarquekeun}
		La restriction $H^3(\cl{Y},\Z_{\ell})$ sans torsion dans le th\'eor\`eme \ref{noyau} n'est pas
		n\'ecessaire. L'hypoth\`ese que $\on{NS}(\cl{Y})$  est sans torsion suffit en effet  \`a assurer que
		l'homomorphisme (\ref{kuenferme}) est  un isomorphisme.
		C'est l\`a un \'enonc\'e assez d\'elicat, pour lequel nous
		renvoyons \`a \cite[Theorem 2.6]{skorobogatov2014brauer} 
		et \`a \cite[Thm. 5.7.7 (iii)]{ctskobrauer}.
	\end{rmk}

	\begin{cor}\label{courbesurfaceqcque} 
		Soit $\F$ un corps fini et $\ell$ un premier distinct de la caract\'eristique de $\F$.
		Soit $S/\F$ une surface projective lisse, g\'eom\'etriquement connexe.
		Soit  $Y/\F$ un produit de courbes projectives, lisses, g\'eom\'etriquement connexes.
		Soit  $X= Y \times S$.
		Alors :
		
		(a) Les propri\'et\'es \'equivalentes de la proposition \ref{surj-char} valent pour $X$.
		
		(b) La fl\`eche $$ \phi_{\ell} : \on{Ker}[CH^2(X)\{\ell\}  \to CH^2(\cl{X})\{\ell\}] \hookrightarrow H^1(\F, H^3(\overline{X},\Z_{\ell}(2))_{\on{tors}})$$ est un isomorphisme de groupes finis.
		
		(c) On a un isomorphisme  
		$$\hspace*{-3cm}\on{Ker}[H^3_{\on{nr}}(\F(X),\Q_{\ell}/\Z_{\ell}(2)) \to H^3_{\on{nr}}(\cl{\F}(X),\Q_{\ell}/\Z_{\ell}(2))]
		$$
		$$\hspace*{5cm} 
		\xrightarrow{\sim} \on{Coker}[CH^2(X)\to CH^2(\cl{X})^G]\{\ell\} .$$
	\end{cor}	
	
	\begin{proof}   Pour une courbe $C$ g\'eom\'etriquement connexe, projective et lisse, on a $\on{NS}(\cl{C})=\Z$ avec action triviale du groupe de Galois, et $H^3(\cl{C},\Z_{\ell})=0$.
		Les conditions \'equivalentes de la proposition \ref{surj-char} sont connues pour une surface
		projective, lisse, g\'eom\'etriquement connexe (voir la remarque \ref{remarquedim2} (iii)).
		Le th\'eor\`eme \ref{noyau} donne (a) par
		produits successifs avec une courbe.
		Il en r\'esulte (b).
		Dans la suite exacte (\ref{ctkahn}), la  fl\`eche injective de groupes finis $$ \on{Ker}[CH^2(X)\{\ell\}  \to CH^2(\cl{X})\{\ell\}] \hookrightarrow H^1(\F, H^3(\overline{X},\Z_{\ell}(2))_{\on{tors}})$$ est donc surjective.
		La suite exacte (\ref{ctkahn}) donne alors (c).
	\end{proof}

	\begin{rmk}
		Soit $X=C\times_{\F}S$,  o\`u $C$ et $S$ sont vari\'et\'es irr\'eductibles, projectives et lisses de dimension $1$ et $2$, respectivement. L'homomorphisme
		\[\on{Ker}[CH^2(S)\{\ell\}  \to CH^2(\cl{S})\{\ell\}]\to \on{Ker}[CH^2(X)\{\ell\}\to CH^2(\cl{X})\{\ell\}]\]
		induit par la projection $X\to S$ n'est pas surjectif en g\'en\'eral. 
		En effet, d'apr\`es le corollaire \ref{courbesurfaceqcque}, la surjectivit\'e de cette fl\`eche \'equivaut \`a celle de
		\[H^1(\F,H^3(\overline{S},\Z_{\ell}(2))_{\on{tors}})\to H^1(\F,H^3(\overline{X},\Z_{\ell}(2))_{\on{tors}}).\]
		Par la formule de K\"unneth en cohomologie $\ell$-adique, le dernier homomorphisme n'est pas surjective d\`es que \[H^1(\F,H^2(\cl{S},\Z_{\ell}(1))_{\on{tors}}\otimes H^1(\cl{C},\Z_{\ell}(1)))\neq 0,\]
		ce qui est le cas, par exemple, si $\ell=2$, $S$ est une surface de Enriques, et la jacobienne de $C$ admet de $2$-torsion non-triviale d\'efinie sur $\F$. 
	\end{rmk}

	\section{Correspondances 
		et   preuve du th\'eor\`eme \ref{mainthm1}}\label{conoyauCH2}
	
	La Nature est un temple o\`u de vivants piliers
	
	Laissent parfois sortir de confuses paroles ;
	
	L'homme y passe \`a travers des for\^{e}ts de symboles
	
	Qui l'observent avec des regards familiers.
	
	(Sonnet des Correspondances, Charles Baudelaire)

	\subsection{Sur un corps alg\'ebriquement clos}\label{coralgclos}
	Soient $k$ un corps alg\'ebriquement clos d'exposant caract\'eristique $p$. Soient $C$ une courbe connexe, projective et lisse sur $k$, $S$ une surface connexe, projective et lisse sur $k$, et $X:=C \times S$.
	
	Notons $p : X \to C$ et $q: X \to S$ les deux projections\footnote{La lettre $p$ est ici utilis\'ee
		dans deux acceptions distinctes, mais aucune confusion n'est possible.}.
	
	On a un homomorphisme
	$$\mu : \on{Pic}(C) \otimes \on{Pic}(S) \xrightarrow{p^*\otimes q^*} \on{Pic}(X)\otimes \on{Pic}(X) \to CH^2(X),$$
	o\`u l'homomorphisme \`a droite est la fl\`eche d'intersection.
	Soit 
	\[\lambda :  CH^2(X)  \to \on{Hom}(\on{Pic}(S), CH_{0}(C)),\qquad \lambda(\Delta):= p_*(q^*(-)\cap \Delta),\] tel que 
	$\lambda(\Delta): \on{Pic}(S) \to CH_{0}(C)$
	est l'homomorphisme standard 
	induit par une correspondance $\Delta$; voir par exemple \cite[Definition 16.1.2]{fulton1998intersection}.
	\begin{comment}
	On a un homomorphisme
	$$ CH^2(X) \otimes \on{Pic}(S) \to \on{Pic}(C)$$ 
	obtenu comme le compos\'e
	$$ CH^2(X) \otimes \on{Pic}(S) \to CH^2(X) \otimes \on{Pic}(X) \to CH^3(X)=CH_{0}(X) \to CH_{0}(C)$$ 
	o\`u  la premi\`ere fl\`eche vient de $q^* : \on{Pic}(S) \to \on{Pic}(X)$,   la seconde fl\`eche de la th\'eorie
	de l'intersection sur $X$, et la troisi\`eme fl\`eche est $p_{*} :CH_{0}(X) \to CH_{0}(C)$.
	Ceci d\'efinit l'homomorphisme $\lambda$.
	\end{comment}
	Enfin, on a un homomorphisme
	$$ \nu  :  \on{Pic}(C) \otimes \on{Pic}(S) \xrightarrow{\mu} CH^2(X)=CH_{1}(X)  \xrightarrow{q_*} CH_{1}(S)=\on{Pic}(S).$$
	
	\begin{lemma}\label{corresp1}
		\begin{enumerate}[label=(\roman*)]
			\item L'application compos\'ee 
			$$\lambda \circ \mu :  \on{Pic}(C) \otimes \on{Pic}(S)  \to CH^2(X) \to \on{Hom}(\on{Pic}(S), CH_{0}(C))$$
			envoie un \'el\'ement $z \otimes D$ sur l'application 
			qui envoie un \'el\'ement $E \in \on{Pic}(S)$
			sur $\on{deg}(D\cdot E)\cdot z$, o\`u $(D\cdot E) \in CH_{0}(S)$.
			\item L'application  $\nu  : \on{Pic}(C) \otimes \on{Pic}(S) \to \on{Pic}(S)$ envoie un \'el\'ement $z\otimes D$
			sur $\on{deg}(z)\cdot D$.	
		\end{enumerate} 
	\end{lemma}
	
	\begin{proof} 
		On utilisera le diagramme cart\'esien \'evident
		\[
		\begin{tikzcd}
			X \arrow[r,"q"] \arrow[d,"p"] & S \arrow[d,"r"]\\
			C \arrow[r,"s"] & \Spec k.    
		\end{tikzcd}	
		\]
		(i) On veut d\'emontrer que l'on a $p_*((p^*(z)\cdot q^*(D))\cdot q^*(D))=\on{deg}(D\cdot E)\cdot z$. Par la formule de projection, on a \[p_*((p^*(z)\cdot q^*(D))\cdot q^*(E))=p_*(p^*(z)\cdot q^*(D\cdot E))=z\cdot p_*q^*(D\cdot E).\] 
		En utilisant le diagramme ci-dessus, on obtient \[p_*q^*(D\cdot E)=s^*r_*(D\cdot E)=s^*(\on{deg}(D\cdot E))=\on{deg}(D\cdot E)[C].\] On conclut que l'on a $z\cdot p_*q^*(D\cdot E)= \on{deg}(D\cdot E)\cdot z$.
		
		(ii) On veut montrer que l'on a $q_*((p^*(z)\cdot q^*(D))=\on{deg}(z)\cdot D$. Par la formule de projection, on a 
		\[q_*((p^*(z)\cdot q^*(D))= q_*p^*(z)\cdot D.\] Par le diagramme ci-dessus, on a 
		\[q_*p^*(z)\cdot D=r^*s_*(z)\cdot D=r^*(\on{deg}(z)) \cdot D =
		\on{deg}(z)\cdot D.\qedhere\]
	\end{proof}

	\begin{prop}\label{corresp2} 
		Soit $k$ un corps alg\'ebriquement clos d'exposant caract\'eristique $p$. Soit $C/k$ une courbe projective et lisse connexe.
		Soit $S$ une surface  projective, lisse sur $k$ et connexe.
		Supposons que $\on{Pic}(S)$ est de type fini et que le discriminant de la forme d'intersection sur $\on{Num}(S)=\on{NS}(S)/\on{tors}$
		est de la forme $\pm p^r$ pour  un entier $r\geq 0$,
		ce qui est le cas si l'on a
		$b_{1}=0$ et  $b_{2}-\rho=0$.
		Soit $X=C \times S$. Alors : 
		
		\begin{enumerate}[label=(\roman*)]
			\item
			La fl\`eche compos\'ee
			$$\on{Pic}(C) \otimes \on{Pic}(S) \to CH^2(X) \xrightarrow{\lambda} \on{Hom}(\on{Pic}(S),\on{Pic}(C))$$
			se factorise comme 
			\[\hspace{-4cm}\on{Pic}(C) \otimes \on{Pic}(S) \to \on{Pic}(C) \otimes \on{Num}(S)  
			\] \[\hspace*{4.5cm}\to \on{Hom}(\on{Num}(S),\on{Pic}(C))\to \on{Hom}(\on{Pic}(S),\on{Pic}(C))\]
			o\`u la premi\`ere fl\`eche est une surjection, la derni\`ere fl\`eche une injection et la fl\`eche m\'ediane 
			a noyau et conoyau finis $p$-primaires. 
			\item Le noyau $N$ de l'homorphisme naturel
			$$\on{Pic}(C) \otimes \on{Pic}(S) \to CH^2(X)$$
			obtenu par image r\'eciproque via chacune des projections, et intersection
			sur $X$, est un groupe fini  $p$-primaire.
			\item Soit $L=k(C)$ le corps des fonctions de la courbe $C$.
			La restriction de la projection  $p : X \to C$ 
			au-dessus du point g\'en\'erique de $C$ induit une fl\`eche surjective $CH^2(X) \to CH^2(S_{L})$
			et une suite exacte 
			\[0 \to  N \to \on{Pic}(C) \otimes \on{Pic}(S) \to CH^2(X) \to CH^2(S_{L}) \to 0.\]
			avec $N$ un groupe fini $p$-primaire.

			\item La fl\`eche $\lambda : CH^2(X) \to \on{Hom}(\on{Pic}(S),\on{Pic}(C))$ induit un homomorphisme 
			$$CH^2(S_{L}) \to \on{Hom}(\on{Pic}(S)_{\on{tors}},\on{Pic}(C)).$$ 
		\end{enumerate} 
	\end{prop}
	
	\begin{proof} 
		On utilise tacitement le fait que $\on{Pic}(C) \simeq \Z \oplus   \on{Pic}^0(C)$, et
		que  pour tout entier $n>0$ la multiplication par $n$ sur $\on{Pic}^0(C)$ 
		est surjective \`a noyau fini.

		(i) Vu le lemme \ref{corresp1}(i), cette fl\`eche compos\'ee se factorise par
		$$\on{Pic}(C) \otimes \on{Num}(S) \to \on{Hom}(\on{Num}(S), \on{Pic}(C)).$$
		L'hypoth\`ese sur $S$ et la forme d'intersection donnent une suite exacte
		$$ 0 \to \on{Num}(S) \to \on{Hom}(\on{Num}(S),\Z)  \to R \to 0$$
		avec $R$ fini $p$-primaire. Si l'on tensorise cette suite exacte par $\on{Pic}(C)$
		et l'on utilise la structure de $\on{Pic}(C)$, on voit que l'homomorphisme
		$$\on{Pic}(C) \otimes \on{Num}(S)  \to \on{Hom}(\on{Num}(S),\Z)  \otimes \on{Pic}(C)= \on{Hom}(\on{Num}(S),\on{Pic}(C))$$ a noyau et conoyau finis
		$p$-primaires.
		Notons   $p^r$ l'exposant du noyau.
		
		(ii)  Soit  $\alpha$ dans  le noyau de $\on{Pic}(C) \otimes \on{Pic}(S) \to CH^2(X)$.
		Ce groupe est  contenu dans le noyau de l'application compos\'ee avec $\lambda$ :
		\[\on{Pic}(C) \otimes \on{Pic}(S) \to \on{Hom}(\on{Pic}(S), \on{Pic}(C)).\]
		L'image de $p^r\cdot\alpha$  dans $\on{Pic}(C) \otimes \on{Num}(S)$
		est nulle.
		
		On a la suite exacte
		\begin{equation}\label{obvious}
			0 \to \on{Pic}(S)_{\on{tors}} \to \on{Pic}(S) \to \on{Num}(S) \to 0
		\end{equation}
		qui est scind\'ee comme suite de groupes ab\'eliens.
		On a donc une suite exacte
		$$ 0 \to \on{Pic}(C) \otimes \on{Pic}(S)_{\on{tors}} \to \on{Pic}(C) \otimes \on{Pic}(S) \to \on{Pic}(C) \otimes \on{Num}(S) \to 0.$$
		Ainsi $\beta= p^r\cdot \alpha$ est dans l'image du groupe fini $\on{Pic}(C) \otimes \on{Pic}(S)_{\on{tors}}   $.
		Fixons un point $P \in  C(k)$. Comme le groupe $\on{Pic}^0(C)$ est divisible, et donc
		$\on{Pic}^0(C) \otimes \on{Pic}(S)_{\on{tors}}=0$,
		l'\'el\'ement  $ \beta$ peut
		s'\'ecrire $n\cdot (P\otimes D)$ avec $D \in \on{Pic}(S)_{\on{tors}}$ et $n$ un entier.
		
		Par hypoth\`ese, l'image de $\alpha$, et donc de $\beta$ dans $CH^2(X)$ est nulle.
		Utilisons maintenant l'application $CH^2(X) \to \on{Pic}(S)$ et le lemme \ref{corresp1}(ii).
		On obtient $n\cdot D=0 \in \on{Pic}(S)$. Donc $\beta=0$ et $p^r\cdot\alpha=0.$
		On conclut que le noyau de $$\on{Pic}(C) \otimes \on{Pic}(S) \to CH^2(X)$$
		est contenu dans le sous-groupe de $\on{Pic}(C) \otimes \on{Pic}(S)$ annul\'e par $p^r$.
		Ce sous-groupe est un groupe fini.

		\medskip
		
		(iii) Pour toute courbe projective et lisse connexe $C$ sur $k$ alg\'ebriquement clos et toute 
		vari\'et\'e projective et lisse connexe $S$ sur $k$, la restriction de la projection $X=C \times S \to C$ 
		au point g\'en\'erique de $C$ induit
		la suite exacte de localisation bien connue
		$$    \oplus_{x\in C^{(1)}}\on{Pic}(S_x) \to  CH^2(X) \to  CH^2(S_L) \to  0, $$
		soit encore
		$$ {\rm Div}(C) \otimes \on{Pic}(S) \to CH^2(X) \to  CH^2(S_L) \to  0.$$
		Elle induit une suite exacte
		$$   \on{Pic}(C)  \otimes \on{Pic}(S) \to CH^2(X) \to  CH^2(S_L) \to  0.$$
		Dans le cas ici consid\'er\'e, on a \'etabli que  la fl\`eche $ \on{Pic}(C)  \otimes \on{Pic}(S) \to CH^2(X)$
		a son noyau fini $p$-primaire.
		
		L'\'enonc\'e (iv) r\'esulte de la suite exacte (\ref{obvious}) et du fait que l'application
		compos\'ee $\on{Pic}(C)  \otimes \on{Pic}(S) \to CH^2(X) \to \on{Hom}(\on{Pic}(S),\on{Pic}(C))$ se factorise
		par  le groupe $\on{Hom}(\on{Num}(S),\on{Pic}(C))$.
	\end{proof}

	\subsection{Sur un corps fini}

	Soit $\g$ un groupe 
	profini, 
	$\h \subset \g$ un sous-groupe normal ouvert. Soit $G=\g/\h$.
	Soient $M$ et $N$ deux $\g$-modules continus discrets. 
	
	On a un homomorphisme de groupes
	ab\'eliens		 	 
	$$\theta_{\h} : M^\h \otimes N^\h \to   M \otimes N$$
	d\'efini par
	$$ (m \otimes n) \mapsto \sum_{g\in G} g(m \otimes n) = \sum_{g\in G} g(m) \otimes g(n).$$
	On v\'erifie imm\'ediatement que l'image de  $\theta_{\h}$ est dans $(M\otimes N)^\g$.
	On consid\`ere l'application 
	$$ \Theta_{M,N} := \oplus_{\h \subset \g} \theta_{\h} : \bigoplus_{\h\subset \g} (M^\h \otimes N^\h)  \to (M\otimes N)^\g.$$

	\begin{lemma}\label{endomi} Soient $\F$ un corps fini et $\g={\rm Gal}(\cl{\F}/\F)$. Soient $M$ un $\g$-module de type fini, et $N$ un $\g$-module.   Supposons que  $N$ satisfait $H^1(\h,N)=0$ pour tout sous-groupe ouvert $\h \subset \g$.
		Alors $\Theta_{M,N}$ est surjective.
	\end{lemma}
	\begin{proof}   
		On commence par l'\'etablir pour  un $\g$-module de la forme $M=\Z[\g/\h]$. 
		C'est un \'enonc\'e g\'en\'eral. Notons que l'on a $M=M^\h$.
		Dans ce cas, pour tout $\g$-module $N$, on a $N^\h \xrightarrow{\sim} (\Z[\g/\h]	 \otimes N)^\g$,
		la fl\`eche \'etant  donn\'ee par $n \mapsto \sum_{\sigma \in G} \sigma \otimes (\sigma\cdot n)$.
		Ceci implique que la fl\`eche  $M^\h \otimes N^\h \to (M\otimes N)^\g$ donn\'ee par
		$$m \otimes n \mapsto \sum_{g\in G} g(m) \otimes g(n)$$ est surjective.
		On en d\'eduit que pour tout $\g$-module $P$ de permutation et tout $\g$-module $N$
		l'application $\Theta_{P,N}$ est surjective.

		Soit maintenant $M$ un $\g$-module de type fini. \begin{comment} C'est un fait facile bien connu (r\'ef\'erence) qu'il existe une suite
		exacte de $g$-modules de type fini
		$$ 0 \to Q \to P \to M \to 0$$
		avec $P$ de permutation et $Q$ un $g$-r\'eseau coflasque, i.e. satisfaisant 
		$H^1(h,Q)=0$ pour tout sous-groupe ouvert $h \subset g$. Comme $\F$ est fini,
		le groupe $g$ est procyclique. D'apr\`es le lemme \ref{endomiyata}, 
		le $g$-module coflasque est  alors un facteur direct d'un module de permutation.
		\end{comment}
		Comme rappel\'e au  lemme \ref{endomiyata}, on a une suite exacte courte de $\g$-modules \[0 \to Q \to P \to M \to 0,\] o\`u $P$ est un $\g$-r\'eseau de permutation and $Q$ est un facteur direct d'un $\g$-r\'eseau de permutation.
		
		Tensorisons la suite exacte ci-dessus par $N$. On obtient une suite exacte
		$$ 0 \to R \to Q \otimes N \to P \otimes N \to M \otimes N \to 0,$$
		o\`u $R$ est un $\g$-module de torsion. Cette suite se coupe en deux suites exactes
		$$ 0 \to R \to 	Q \otimes N \to S \to 0$$
		et 
		$$0 \to S \to P \otimes N \to M \otimes N \to 0.$$
		La premi\`ere donne la suite exacte
		$$ H^1(\g, Q \otimes N) \to H^1(\g,S) \to  H^2(\g,R).$$
		L'hypoth\`ese faite sur $N$ dans le lemme, le fait que $Q$ est un facteur direct
		d'un module de permutation, et le lemme de Shapiro, donnent $H^1(\g, Q \otimes N) =0$.
		Par ailleurs, comme $\F$ est un corps fini et $R$ est de torsion, on a $H^2(\g,R)=0$.
		On a donc $H^1(\g,S) =0$. La deuxi\`eme suite exacte donne alors que
		l'application $	(P \otimes N)^\g  \to (M \otimes N )^\g$ est surjective.
		
		L'homomorphisme $P\to M$ induit le diagramme commutatif suivant:
		\[
		\begin{tikzcd}
			\bigoplus_{\h\subset \g} (P^\h \otimes N^\h) \arrow[d]  \arrow[r,"\Theta_{P,N}"] & (P\otimes N)^\g \arrow[d,->>]\\ 
			\bigoplus_{\h\subset \g} (M^\h \otimes N^\h)  \arrow[r,"\Theta_{M,N}"] &  (M\otimes N)^\g.
		\end{tikzcd}
		\]
		Le fait que $\Theta_{P,N}$ est surjectif implique donc que $\Theta_{M,N}$ est surjectif.
	\end{proof}

	\begin{prop}\label{corresp2fini} 
		Soit $\F$ un corps fini de caract\'eristique $p$.  Soit $G={\rm Gal}(\cl{\F}/\F)$. 
		Soit $C/\F$ une courbe projective, lisse  et g\'eom\'etriquement connexe.
		Soit $K=\F(C)$ et $L=\cl{\F}(C)$.
		Soit $S/\F$ une surface  projective, lisse   et g\'eom\'etriquement connexe.
		Supposons que $\on{Pic}(\cl{S})$ est de type fini et que le discriminant de la   forme d'intersection sur  le r\'eseau 
		$\on{Num}(\cl{S})=\on{NS}(\cl{S})    /  \on{NS}(\cl{S})_{\rm tors}       $
		est de la forme $\pm p^r$ pour  un entier $r\geq 0$,
		ce qui est le cas si l'on a
		$b_{1}=0$ et  $b_{2}-\rho=0$.
		Soit $X=C \times_{\F} S$. Alors : 
		
		(i) La suite exacte $$ 0 \to (\on{Pic}(\cl{S}) \otimes \on{Pic}(\cl{C} ))_{{p}} \to CH^2(\cl{X})_{{p}} \to CH^2(S_{L})_{{p}} \to 0$$
		obtenue \`a partir de celle de la proposition \ref{corresp2} par tensorisation par $\Z[1/p]$ induit une suite exacte
		$$ 0 \to (\on{Pic}(\cl{S}) \otimes \on{Pic}(\cl{C} ))_{{p}}^G \to CH^2(\cl{X})_{{p}}^G  \to CH^2(S_{L})_{{p}}^G \to 0.$$
		
		(ii) L'image de $(\on{Pic}(\cl{S}) \otimes \on{Pic}(\cl{C} ))^G \to CH^2(\cl{X})^G$
		est incluse dans l'image de $CH^2(X) \to CH^2(\cl{X})^G$.
		
		(iii) La restriction $CH^2(\cl{X}) \to CH^2(S_{L})$ induit un isomorphisme de groupes de torsion
		$$\on{Coker}[CH^2(X)\to CH^2(\cl{X})^G]_{{p}} \xrightarrow{\sim} \on{Coker}[CH^2(S_K)\to CH^2(S_{L})^G]_{{p}}$$ et ce dernier groupe
		est isomorphe \`a $\on{Coker}[A_{0}(S_{K})\to A_{0}(S_{L})^G]_{{p}}$.
	\end{prop}
	
	\begin{proof}  
		Sur $\cl{X}= \cl{C} \times \cl{S}$, les correspondances consid\'er\'ees dans la sous-section  \ref{coralgclos}
		induisent des homomorphismes $G$-\'equivariants.
		
		On consid\`ere la suite de cohomologie galoisienne associ\'ee \`a la suite	 tensoris\'ee par $\Z[1/p]$.
		Notons $M$ le noyau de $CH^2(\cl{X})\to CH^2(S_L)$. Par la suite de la proposition \ref{corresp2}(iii), on a $ (\on{Pic}(\cl{C})\otimes \on{Pic}(\cl{S}))_{{p}} \xrightarrow{\sim} M_{{p}}$. Ceci donne la premi\`ere suite exacte de (i). Par passage \`a la cohomologie galoisienne, on d\'eduit formellement une suite exacte
		$$ 0 \to (\on{Pic}(\cl{S}) \otimes \on{Pic}(\cl{C} ))_{{p}}^G \to CH^2(\cl{X})_{{p}}^G \to CH^2(S_{L})_{{p}}^G$$
		et le fait que la derni\`ere fl\`eche est surjective si et seulement si la fl\`eche
		$$H^1(G,  \on{Pic}(\cl{S}) \otimes \on{Pic}(\cl{C} ) )_{{p}} \to H^1(G, CH^2(\cl{X}))_{{p}}$$ 
		est injective. Ici on utilise le fait que, comme $\Z[1/p]$ est $\Z$-plat, la cohomologie et les $G$-invariants en particulier commutent avec tensorisation par $\Z[1/p]$.
		On a  :
		$$ \on{Pic}(\cl{S}) \otimes \on{Pic}(\cl{C} )=   \on{Pic}(\cl{S}) \oplus [ \on{Pic}(\cl{S}) \otimes \on{Pic}^0(\cl{C} )].$$
		On va maintenant utiliser le fait que le corps de base $\F$ est fini.
		Pour la courbe $C$ projective et lisse sur le corps fini $\F$, on a $H^1(\F, R \otimes  \on{Pic}^0(\cl{C} ))=0$ pour tout module galoisien de type fini $R$
		(cons\'equence du th\'eor\`eme de Lang).
		La fl\`eche  $ H^1(\F, \on{Pic}(\cl{S}) ) \to   H^1(\F,  \on{Pic}(\cl{S}) \otimes \on{Pic}(\cl{C} ))$ d\'efinie par le choix d'un z\'ero-cycle de degr\'e 1 sur $C$ est donc un isomorphisme.
		L'application compos\'ee  $$\on{Pic}(\cl{S}) \to \on{Pic}(\cl{S}) \otimes \on{Pic}(\cl{C} ) \to CH^2(\cl{X})  \to \on{Pic}(\cl{S}), $$ o\`u la troisi\`eme fl\`eche est
		induite par la projection $X \to S$, est l'identit\'e. On conclut que 
		$$H^1(\F,  \on{Pic}(\cl{S}) \otimes \on{Pic}(\cl{C} )) \to H^1(\F, CH^2(\cl{X}) )$$ est injectif, ce qui \'etablit (i).
		
		Appliquons maintenant  le lemme \ref{endomi} \`a
		$M=\on{Pic}(\cl{S})$ et $N= \on{Pic}(\cl{C})$. 
		Comme $\F$ est un corps fini, l'hypoth\`ese  $H^1(\F',\on{Pic}(\cl{C}))=0$ pour toute extension finie $\F'/\F$
		est satisfaite.
		Par ailleurs,
		pour $\F'/\F$ extension finie, 
		le sous-groupe $M_{\F'}$ de $M$ form\'e des invariants sous ${\rm{Gal}(\cl{\F}/\F'})$ est \'egal 
		\`a $\on{Pic}(S_{\F'})$. De m\^eme pour $N= \on{Pic}(\cl{C})$.
		D'apr\`es le lemme, l'image de
		$(\on{Pic}(\cl{S}) \otimes \on{Pic}(\cl{C} ))^G $ dans $CH^2(\cl{X})^G$ co\"{\i}ncide avec la r\'eunion des images
		des $M_{\F'} \otimes N_{\F'} $ par l'application compos\'ee de 
		$M_{\F'} \otimes N_{\F'}  \to CH^2(X_{\F'})$ et la norme, c'est-\`a-dire 
		la somme $\sum_{\sigma \in {\rm Gal}(\F'/\F)}	 \sigma$.
		L'application $$\sum_{\sigma \in {\rm Gal}(\F'/\F)} \sigma  : CH^2(X_{\F'})  \to CH^2(X_{\F'})$$ se factorise
		comme la compos\'ee de la norme $CH^2(X_{\F'}) \to CH^2(X)$ et de la fl\`eche
		de restriction $CH^2(X) \to CH^2(X_{\F'})$.  Ceci \'etablit (ii). 
		
		Comme $\F$ est un corps fini, la surface g\'eom\'etriquement int\`egre 
		$S/\F$ poss\`ede un z\'ero-cycle de degr\'e 1 (Lang--Weil).
		Comme la restriction 
		$CH^2(X) \to CH^2(S_{K})$ est surjective, l'\'enonc\'e (iii) r\'esulte des  deux autres \'enonc\'es.
	\end{proof}
	
	\begin{lemma}\label{albanese}
		Soient $k$ un corps et $X/k$ une vari\'et\'e projective, lisse et g\'eom\'etrique\-ment connexe. Supposons que $A_0(X_{\Omega})\otimes\Q=0$ pour toute extension alg\'ebriquement close $\Omega/k$. Alors $A_0(X_{\Omega})=0$ pour toute extension alg\'ebriquement close $\Omega/k$.
	\end{lemma}
	
	\begin{proof}
		Soit $A$ la vari\'et\'e d'Albanese de $X$. Pour tout corps alg\'ebriquement clos $\Omega$ on a un homomorphisme surjectif $A_0(X_{\Omega})\to A(\Omega)$. Par hypoth\`ese, pour tout corps alg\'ebriquement clos $\Omega$, on a $A_{0}(X_{\Omega})\otimes \Q=0$, et donc $A(\Omega)\otimes \Q=0$. Comme $A$ est une vari\'et\'e ab\'elienne, si $A\neq 0$ il existe $\Omega$ tel que $A(\Omega)$ n'est pas de torsion. Par exemple, si on prend $\Omega$ contenant $k(A)$, le point de $A(\Omega)$ correspondant \`a l'identit\'e de $A$ n'est pas de torsion. On en d\'eduit que $A=0$. Par le th\'eor\`eme de Roitman \cite{roitman, blochroitman, milneroitman}, l'homomorphisme $A_0(X_{\Omega})\to A(\Omega)$ induit un isomorphisme sur la torsion. Comme $A(\Omega)=0$, on obtient que $A_0(X_{\Omega})$ est sans torsion. D'apr\`es \cite[Lemma 1.3]{bloch1976some}, le groupe $A_{0}(X_{\Omega})$ est divisible, et donc uniquement divisible. L'hypoth\`ese $A_0(X_{\Omega})\otimes \Q=0$ entra\^{\i}ne alors   $A_0(X_{\Omega})=0$.
	\end{proof}

	\begin{thm}\label{presquemainthm}  
		Soit $\F$ un corps fini de caract\'eristique $p$.  Soit $G={\rm Gal}(\cl{\F}/\F)$. 
		Soit $C/\F$ une courbe projective, lisse  et g\'eom\'etriquement connexe.
		Soit $K=\F(C)$ et $L=\cl{\F}(C)$.
		Soit $S/\F$ une surface projective, lisse et g\'eom\'etriquement connexe satisfaisant 
		$b_{1}=0$ et $b_{2}-\rho=0$, et soit $X:=S\times C$.  
		Avec les notations de la proposition \ref{corresp2fini},
		les \'enonc\'es
		(a), (b) (c)  
		suivants
		sont \'equivalents :
		
		(a)  On a $\on{Ker}[H^3_{\on{nr}}(\F(X),\Q_{\ell}/\Z_{\ell}(2)) \to H^3_{\on{nr}}(\cl{\F}(X),\Q_{\ell}/\Z_{\ell}(2))] =0$.
		
		(b)  Le conoyau, de torsion,  de l'application  naturelle  de groupes   $A_{0}(S_{K}) \to A_{0}(S_{L})^G$
		n'a pas de torsion $\ell$-primaire.
		
		Supposons de plus que la surface $S$ est g\'eom\'etriquement $CH_{0}$-triviale.
		Alors les conditions ci-dessus sont \'equivalentes \`a :
		
		(c) La fl\`eche
		$A_{0}(S_{\F(C)})\{\ell\} \to A_{0}(S_{\cl{\F}(C)})\{\ell\}^G$
		est surjective.
	\end{thm}

	\begin{proof} 
		Par le corollaire \ref{courbesurfaceqcque}, on a un isomorphisme de groupes finis
		\[\on{Ker}[CH^2(X)\to CH^2(\cl{X})]\{\ell\}\xrightarrow{\sim}  H^1(\F, H^3(\cl{X},\Z_{\ell}(2))_{\rm tors} ).\] 
		Sans m\^eme avoir \`a  identifier les fl\`eches, un argument de comptage et la suite exacte
		(\ref{ctkahn}) donnent un isomorphisme
		$$\hspace*{-5cm}\on{Ker}[H^3_{\on{nr}}(\F(X),\Q_{\ell}/\Z_{\ell}(2)) \to H^3_{\on{nr}}(\cl{\F}(X),\Q_{\ell}/\Z_{\ell}(2))]$$
		$$\hspace*{6.8cm}\xrightarrow{\sim} \on{Coker}[CH^2(X)\to CH^2(\cl{X})^G]\set{\ell},$$
		et par la proposition \ref{corresp2fini}  ce  groupe est isomorphe \`a  $\on{Coker}[A_{0}(S_{K})\to A_{0}(S_{L})^G]\{\ell\}$.	
		Ceci donne l'\'equivalence de (a) et (b).
		Si la surface $S$ est 	g\'eom\'etriquement $CH_{0}$-triviale, alors les groupes 
		$A_{0}(S_{\F(C)})$ et $A_{0}(S_{\cl{\F}(C)})$ sont de torsion, ce qui donne l'\'equivalence de (b) et (c).
	\end{proof}

	\begin{thm}\label{presquemainthmTate}
		Supposons  que la conjecture de Tate vaut pour toutes les surfaces sur les corps finis. 
		Avec les notations et hypoth\`eses du th\'eor\`eme \ref{presquemainthm}, sous l'hypoth\`ese que
		la surface $S$ est g\'eom\'etriquement $CH_{0}$-triviale,
		les \'enonc\'es (a), (b) ou (c) impliquent:

		Le groupe $H^3_{\on{nr}}(\F(X),\Q_{\ell}/\Z_{\ell}(2))$ est nul,  et l'application cycle $CH^2(X) \otimes \Z_{\ell} \to H^4(X,\Z_{\ell}(2))$ est surjective.
	\end{thm}
	
	\begin{proof} 		
		Soit $\iota: \cl{C}\hookrightarrow \cl{C}\times \cl{S}$ une immersion ferm\'ee associ\'ee \`a un $\cl{\F}$-point de $\cl{S}$. Soit $\Omega$ un corps alg\'ebriquement clos contenant $\F$. Par le lemme \ref{albanese}, on a $A_0(S_{\Omega})=0$. 
		On d\'eduit ais\'ement que l'homomorphisme ${CH}_0(C_{\Omega})\to {CH}_0(X_{\Omega})$ induit par $\iota$ est surjectif.
		
		Par le corollaire \ref{dim3finialgclos} ci-dessous, qui utilise un th\'eor\`eme de C. Schoen (c'est ici qu'on utilise la conjecture de Tate pour les surfaces), on a alors  $$H^3_{\on{nr}}(\cl{\F}(X),\Q_{\ell}/\Z_{\ell}(2))=0.$$ Sous l'hypoth\`ese (a) on a donc $$H^3_{\on{nr}}(\F(X),\Q_{\ell}/\Z_{\ell}(2))=0.$$  Par ailleurs, la surjectivit\'e de  ${CH}_0(C_{\Omega})\to {CH}_0(X_{\Omega})$ et  
		\cite[Prop. 3.23]{colliot2013cycles} donnent que  
		l'application cycle
		$CH^2(X) \otimes \Z_{\ell} \to H^4(X,\Z_{\ell}(2))$ pour le solide $X$
		a  son conoyau fini.  D'apr\`es le th\'eor\`eme \ref{KCTK} rappel\'e ci-dessous, ce 
		conoyau
		est alors \'egal
		\`a un quotient de $H^3_{\on{nr}}(\F(X),\Q_{\ell}/\Z_{\ell}(2))$, et donc est  nul. 
		On a donc montr\'e que  l'application cycle
		\[ CH^2(X) \otimes \Z_{\ell} \to H^4(X,\Z_{\ell}(2))\]
		est surjective.
	\end{proof}

	\begin{thm}\label{t}
		Soient $k$ un corps de caract\'eristique $p\geq 0$, $\cl{k}$ une cl\^oture s\'eparable de $k$, $S$ une $k$-surface projective et lisse g\'eom\'etriquement int\`egre, et $\cl{S}:=S\times_k\cl{k}$. Supposons que $S$ est g\'eom\'etriquement $CH_{0}$-trivale et que  $S(k)\neq \emptyset$. 
		Soit $m\geq 1$ un entier tel que $m\cdot \on{NS}(\cl{S})_{\rm tors}=0$.
		Soit $n\geq 1$ le  degr\'e d'une extension finie galoisenne $K/k$
		telle qu'apr\`es extension \`a $K$ le module galoisien $\on{NS}(\cl{S})/{\rm tors}$
		soit de permutation. Soit $N=mn$.

		Pour tout corps $F$ contenant $k$, on  a  $N \cdot A_{0}(S_{F})=0$.
	\end{thm}
	Le th\'eor\`eme \ref{t} et son corollaire \ref{t-cor} seront utilis\'es seulement dans le cas o\`u $k$ est alg\'ebriquement clos.
	
	\begin{proof}  Pour $k$  alg\'ebriquement clos de caract\'eristique quelconque, le r\'esultat
		est  \cite[Cor. 6.4 (a)]{kahn2017torsion}. Pour $k$ quelconque, il convient
		de suivre les d\'emonstrations du paragraphe 3 de  \cite{colliot1985k}.
		C'est fait dans \cite[Th\'eor\`eme A.1(b), Remarque A.2(3)]{kahn2017torsion} 
		sur un corps de caract\'eristique z\'ero. Cette restriction sur la caract\'eristique  est inutile
		car  
		les r\'esultats de  \cite{colliot1985k}
		sont \'etablis 	en toute caract\'eristique $p>0$ pour tout $\ell\neq p$. .
	\end{proof}
	
	\begin{cor}\label{t-cor}
		Sous les hypoth\`eses du th\'eor\`eme
		\ref{t} pour la surface $S$,
		soient $V$ une $k$-vari\'et\'e projective et lisse g\'eom\'etriquement int\`egre, $X:=V\times_kS$, $P\in S(k)$, et $f : V \hookrightarrow X$ donn\'e par $f(x)=(x,P)$.
		
		(i) Pour tout corps $F$ contenant $k$, on a
		\[N \cdot f_{*}(CH_{0}(V_{F}) ) = N \cdot CH_{0}(X_{F}).\]
		
		(ii) Si $H^{i}(k,\Q/\Z(j)) \to H^{i}_{\on{nr}}(k(V),\Q/\Z(j))$ est un isomorphisme,
		alors, pour tout $\ell$ qui ne divise pas $N $, l'homomorphisme \[H^{i}(k,\Q_{\ell}/\Z_{\ell}(j))  \to H^{i}_{\on{nr}}(k(X),\Q_{\ell}/\Z_{\ell}(j))\] est surjectif. 
	\end{cor}
	
	\begin{proof}
		(i) Il suffit de d\'emontrer que $N \cdot CH_{0}(X_{F})\subset N \cdot f_{*}(CH_{0}(V_{F}) )$. Soit  $D$ un z\'ero-cycle de $X_F$, 
		et
		soit $d$ le degr\'e de $D$. Par le th\'eor\`eme \ref{t}, $N(D-dP)=0$ dans $A_0(S_F)$, donc $N D=NdP$ est dans $N  f_{*}(CH_{0}(V_{F}))$. 
		
		(ii) Ceci suit de (i) en utilisant un argument comme dans \cite[Thm. 1.4]{auelctparimala}. %ref?
	\end{proof}

	\begin{proof}[Preuve du th\'eor\`eme \ref{mainthm1}]
		Comme $\on{NS}(\cl{S})[\ell]=0$ et $S$ est g\'eom\'etriquement $CH_0$-triviale, on peut appliquer le th\'eor\`eme \ref{t} et le corollaire \ref{t-cor}(ii) avec $V=C$, $k=\cl{\F}$, $i=3$ et $j=2$
		et $F=\cl{\F}(C)$.
		On obtient que $A_0(S_{\cl{\F}(C)})\set{\ell}=0$ et, comme la $\ell$-dimension cohomologique de $\cl{\F}(C)$ est~$1$, $H^3_{\rm nr} (\cl{\F}(X),\Q_{\ell}/\Z_{\ell}(2))=0$. Par le th\'eor\`eme \ref{presquemainthm}, on d\'eduit que $H^3_{\on{nr}}(\F(X),\Q_{\ell}/\Z_{\ell}(2))=0$. 
		
		Par ailleurs, la fl\`eche  ${CH}_0(C_{\Omega})\to {CH}_0(X_{\Omega})$ est surjective pour tout corps alg\'ebriquement clos $\Omega$ contenant $\F$. Donc,   
		par \cite[Prop. 3.23]{colliot2013cycles} le conoyau
		$CH^2(X) \otimes \Z_{\ell} \to H^4(X,\Z_{\ell}(2))$ 
		est fini.  D'apr\`es le th\'eor\`eme \ref{KCTK} rappel\'e ci-dessous, ce conoyau est isomorphe
		\`a un quotient de $H^3_{\on{nr}}(\F(X),\Q_{\ell}/\Z_{\ell}(2))$, et donc il est nul. 
	\end{proof}

	\section{Z\'ero-cycles sur $S_{\cl{\F}(C)} $ et preuve des Th\'eor\`emes \ref{mainthm2} et  \ref{mainthm3}}\label{ktheoriedure}\label{conoyauCH2bis}

	\begin{thm}\label{a0complete}(Raskind)
		Soient $k$ un corps d'exposant caract\'eristique $p$ et $Z$ une $k$-vari\'et\'e projective, lisse, g\'eom\'etriquement connexe. 
		On a les propri\'et\'es suivantes :
		\begin{enumerate}[label=(\alph*)]
			\item L'application de restriction $CH^2(Z_{k[\![t]\!]})\to CH^2(Z_{k(\!(t)\!)})$ est un isomorphisme.
			
			\item Pour chaque premier $\ell\neq p$, on a une injection
			\[CH^2(Z_{\cl{k}[\![t]\!]})\set{\ell}\hookrightarrow CH^2(Z_{\cl{k}})\set{\ell},\] 
			
			\item Pour chaque premier $\ell\neq p$, on a une injection
			\[CH^2(Z_{\cl{k}(\!(t)\!)})\set{\ell}\hookrightarrow CH^2(Z_{\cl{k}})\set{\ell},\] 
		\end{enumerate} 
	\end{thm}
	
	\begin{proof} 
		Par \cite[Proposition 1.2]{raskind1989torsion}, on a une suite exacte
		\[H^1_{\on{Zar}}(Z_{k(\!(t)\!)},\mc{K}_{2}) \to \on{Pic}(Z) \to CH^2(Z_{k[\![t]\!]}) \to CH^2(Z_{k(\!(t)\!)}) \to 0.\]
		Si $[\mc{L}]\in \on{Pic}(Z)$, la fl\`eche 
		\[\on{Pic}(Z) \otimes k(\!(t)\!)^* \to H^1_{\on{Zar}}(Z_{k(\!(t)\!)},\mc{K}_{2}) \to \on{Pic}(Z)\]
		envoie $[\mc{L}]\otimes t$ sur $\mc{L}$; voir le carr\'e commutatif dans la preuve de \cite[Proposition 1.3]{raskind1989torsion}.
		La fl\`eche $H^1_{\on{Zar}}(Z_{k(\!(t)\!)},\mc{K}_{2}) \to \on{Pic}(Z)$ est donc surjective.
		On a donc
		$CH^2(Z_{k[\![t]\!]}) =CH^2(Z_{k(\!(t)\!)})$.
		Par \cite[Theorem 1.9]{raskind1989torsion},  pour chaque premier $\ell\neq p$, on a une injection
		\[CH^2(Z_{\cl{k}[\![t]\!]})\set{\ell}\hookrightarrow CH^2(Z_{\cl{k}})\set{\ell}.\] 
		La combinaison de ces r\'esultats donne le th\'eor\`eme.
	\end{proof}

	On va utiliser le th\'eor\`eme \ref{general} ci-dessous, dont la d\'emonstration repose sur les r\'esultats 
	de Merkurjev et Suslin \cite{merkurjev1982k-cohomology}, sur
	le th\'eor\`eme 90 de Hilbert pour $K_{2}$ \cite{colliot1983hilbert}, et  sur de nombreux r\'esultats du travail \cite{colliot1985k}
	de Raskind et du premier auteur.
	En caract\'eristique z\'ero, on trouvera
	une d\'emonstration d\'etaill\'ee de ce th\'eor\`eme   dans
	l'article \cite{ctvoisin}: c'est le th\'eor\`eme 8.7 de \cite{ctvoisin}, o\`u
	l'hypoth\`ese $b_{2}-\rho=0$ est remplac\'ee par 
	l'hypoth\`ese $H^2(X,\mc{O}_{X})=0$, qui lui est \'equivalente
	en caract\'eristique nulle. 
	En caract\'eristique $p>0$, si on prend comme hypoth\`ese
	$b_{2}-\rho=0$ on voit en suivant  
	les divers arguments donn\'es dans \cite[thm. 8.7]{ctvoisin} et en se r\'ef\'erant 
	\`a \cite{colliot1985k} (voir en particulier \cite[Rem. 2.14]{colliot1985k} et \cite[Remark 3.7.1]{colliot1985k})
	que tout vaut \`a la $p$-torsion pr\`es. 
	Sur un corps parfait, le th\'eor\`eme  est  aussi \'etabli dans \cite[Thm. 6.3]{colliot2013cycles} et \cite[Thm. 6.6]{colliot2013cycles}. 
	C'est d'ailleurs ainsi que sur un corps fini la suite exacte (\ref{ctkahn}) est \'etablie dans \cite{colliot2013cycles}.
	On notera  que si l'on prend pour $L$ un corps fini,
	on a la suite exacte d\'ecrite ci-dessous sans supposer
	$b_{2}-\rho=0$ .
	Sur un corps $L$ parfait de caract\'eristique $p>0$,  la $p$-torsion est
	aussi contr\^{o}l\'ee (Gros et Suwa \cite{grossuwa}).

	\begin{thm}\label{general}  Soit $L$ un corps de caract\'eristique $p\geq 0$,  et $\ell \neq p$ un premier.
		Supposons $\on{cd}_{\ell}(L) \leq 1$. Soit $\cl{L}$ une cl\^oture s\'eparable de $L$, et $G={\rm Gal}(\cl{L}/L)$.
		Soit $X$ une $L$-vari\'et\'e projective, lisse, g\'eom\'etriquement connexe poss\'edant un 
		z\'ero-cycle de degr\'e 1.
		Supposons que le rang $\rho$ du groupe de N\'eron-Severi de $\cl{X}$
		est \'egal au deuxi\`eme nombre de Betti $\ell$-adique $b_{2}
		=\dim H^2(\cl{X}, \Q_{\ell }  )$. On a alors une suite exacte naturelle
		$$\hspace*{-2cm} 0 \to \on{Ker}[CH^2(X)\{\ell\} \to CH^2(\cl{X})\{\ell\} ] \to H^1(G,   H^3(\cl{X},\Z_{\ell}(2))\{\ell\}   )   $$
		$$\hspace*{1cm}\to \on{Ker}[H^3_{\on{nr}}(L(X)/L, \Q_{\ell}/\Z_{\ell}(2)) \to H^3_{\on{nr}}(\cl{L}(X)/\cl{L}, \Q_{\ell}/\Z_{\ell}(2)) ]$$
		$$\hspace*{6cm}\to \on{Coker}[CH^2(X)  \to CH^2(\cl{X})^G] \{ \ell\}  \to 0.$$ 
	\end{thm}
	
	\begin{cor}\label{maincorctr}
		Soit $L$ un corps de caract\'eristique $p\geq 0$,  et $\ell \neq p$ un premier.
		Supposons $\on{cd}_{\ell}(L) \leq 1$. Soit $\cl{L}$ une cl\^oture s\'eparable de $L$, et $G={\rm Gal}(\cl{L}/L)$.
		Soit $S$ une $L$-surface projective, lisse, g\'eom\'etriquement connexe.
		
		(a) Si l'on a $b_{1}=0$ et $b_{2}-\rho=0$, 
		alors on a une suite exacte naturelle
		$$ 0 \to A_{0}(S)\{\ell\}   \to H^1(G,   H^3(\cl{S},\Z_{\ell}(2))\{\ell\}   )    
		\to  H^3_{\on{nr}}(L(S)/L, \Q_{\ell}/\Z_{\ell}(2)).$$
		
		(b) Si de plus $S$ est
		g\'eom\'etriquement $CH_{0}$-triviale et poss\`ede un 
		z\'ero-cycle de degr\'e 1,
		alors on a une suite exacte naturelle
		$$ 0 \to A_{0}(S)\{\ell\}   \to H^1(G,   H^3(\cl{S},\Z_{\ell}(2))\{\ell\}   )    
		\to  H^3_{\on{nr}}(L(S)/L, \Q_{\ell}/\Z_{\ell}(2))  \to 0.$$
	\end{cor}
	
	\begin{proof} 
		Comme $S$ est  une surface, le corps $\cl{L}(S)$ est de $\ell$-dimension cohomo\-logique~2,
		donc $H^3(\cl{L}(S), \Q_{\ell}/\Z_{\ell}(2))=0$.  Si $b_{1}=0$, alors $A_{0}(\cl{S})\{\ell\}=0$
		par le th\'eor\`eme de Roitman \cite{roitman, blochroitman}. 
		L'\'enonc\'e  (a) r\'esulte alors du th\'eor\`eme \ref{general}.

		Si la surface $S$ est g\'eom\'etriquement $CH_{0}$-triviale, alors, d'apr\`es le lemme \ref{albanese},
		la fl\`eche degr\'e $CH_{0}(\cl{S}) \to \Z$ est
		surjective \`a noyau de torsion $p$-primaire, et 
		l'existence d'un z\'ero-cycle de degr\'e 1 sur   
		la surface $S$ donne alors
		$$\on{Coker}[CH^2(S)  \to CH^2(\cl{S})^G] \{ \ell\} =0.$$
		L'\'enonc\'e (b) r\'esulte alors du th\'eor\`eme \ref{general}.
	\end{proof}

	\begin{thm}\label{a0}
		Soient $k$ un corps parfait 
		d'exposant  caract\'eristique $p$ et $\cl{k}/k$ une cl\^{o}ture s\'eparable. Soient   $C$ et $S$ des $k$-vari\'et\'es projectives, lisses, g\'eom\'e\-tri\-que\-ment connexes de dimension $1$ et $2$, respectivement. Soit $X= S\times_{k}C$. Notons $K:=k(C)$ et $L:=\cl{k}(C)$.  
		
		(a) Si la surface $S$ satisfait $b_{1}=0$ et $b_{2}-\rho=0$, alors pour tout premier $\ell\neq p$, on a une suite exacte \footnote{Nous ne savons pas si la fl\`eche $A_0(S_L)\set{\ell}\to 
			{\rm Hom}({\rm Pic}(\cl{S})\{\ell\}, J(C)(\cl{k}))$, qui est d\'efinie par la $K$-th\'eorie alg\'ebrique, est induite par
			la fl\`eche $CH^2(S_{L}) \to {\rm Hom} ({\rm Pic}(\cl{S})_{\on{tors}}, {\rm Pic}(\cl{C}))$ obtenue par
			les correspondances \`a la proposition \ref{corresp2} (iv).}
		$\on{Gal}(\cl{k}/k)$-\'equivariante de groupes finis \[0 \to A_0(S_L)\set{\ell}\to 
		{\rm Hom}_{\Z}({\rm Pic}(\cl{S})\{\ell\}, J(C)(\cl{k})) \to H^3_{\rm nr} (\cl{k}(X)/\cl{k},\Q_{\ell}/\Z_{\ell}(2)) .\] 
		
		(b) Si de plus $S$ est
		g\'eom\'etriquement $CH_{0}$-triviale et poss\`ede un 
		z\'ero-cycle de degr\'e 1, alors, pour tout premier $\ell\neq p$, on a une suite exacte
		$\on{Gal}(\cl{k}/k)$-\'equivariante de groupes finis \[0 \to A_0(S_L)\set{\ell}\to 
		{\rm Hom}_{\Z}({\rm Pic}(\cl{S})\{\ell\}, J(C)(\cl{k})) \to H^3_{\rm nr} (\cl{k}(X)/\cl{k},\Q_{\ell}/\Z_{\ell}(2)) \to 0.\] 
	\end{thm}

	\begin{proof}
		Soit  $L=\cl{k}(C)$. Pour chaque $x\in \cl{C}^{(1)}$, on note par $L_x \simeq \cl{k}(\!(t)\!)$ le corps des fractions de l'anneau local compl\'et\'e de $\cl{C}$ en $x$.
		
		Donnons la d\'emonstration de (b). La d\'emonstration de (a) est identique, on enl\`eve simplement
		le z\'ero \`a droite dans les suites exactes ci-dessous. Le th\'eor\`eme de Sato et Saito cit\'e ci-dessous
		n'intervient que pour la partie (b).

		On applique le corollaire \ref{maincorctr}
		\`a $S_L$ et $S_{L_x}$ pour chaque $x\in \cl{C}^{(1)}$. Comme $\on{cd}(L)=\on{cd}(L_{x})=1$, pour tout $\ell\neq p$ on obtient le diagramme commutatif suivant:
		\[
		\adjustbox{scale=0.685,center}{ 
			\begin{tikzcd}
				0 \arrow[r] & A_0(S_L)\set{\ell} \arrow[r] \arrow[d]  &  H^1(L,H^3(\cl{S},\Z_{\ell}(2))_{\on{tors}}) \arrow[r] \arrow[d] & H^3_{\on{nr}}(L(S)/L,\Q_{\ell}/\Z_{\ell}(2)) \arrow[r] \arrow[d] & 0 \\
				0 \arrow[r] & \prod_{x\in \cl{C}^{(1)}} A_0(S_{L_{x}})\set{\ell} \arrow[r] & \prod_{x\in \cl{C}^{(1)}} H^1(L_{x},H^3(\cl{S},\Z_{\ell}(2))_{\on{tors}}) \arrow[r] & \prod_{x\in \cl{C}^{(1)}} H^3_{\on{nr}}(L_{x}(S)/L_{x},\Q_{\ell}/\Z_{\ell}(2))\arrow[r] & 0.  
			\end{tikzcd}
		}
		\]
		L'hypoth\`ese $b_{1}=0$ et le th\'eor\`eme de Roitman \cite{roitman, blochroitman} donnent $A_{0}(\cl{S})\{\ell\}=0$.
		D'apr\`es le th\'eor\`eme \ref{a0complete}, on a donc $A_0(S_{L_{x}})\set{\ell}=0$ pour tout $x\in \cl{C}^{(1)}$.

		Le noyau de \[H^3_{\on{nr}}(L(S)/L,\Q_{\ell}/\Z_{\ell}(2))\to \prod_{x\in \cl{C}^{(1)}} H^3_{\on{nr}}(L_{x}(S)/L_{x},\Q_{\ell}/\Z_{\ell}(2))\] est un sous-groupe de $H^3_{\on{nr}}(\cl{k}(X)/\cl{k},\Q_{\ell}/\Z_{\ell}(2))$, comme on voit en consid\'erant les r\'esidus aux points g\'en\'eriques des diviseurs $S_{x}  \subset X$.
		Ce noyau co\"{\i}ncide en fait avec $H^3_{\on{nr}}(\cl{k}(X)/\cl{k},\Q_{\ell}/\Z_{\ell}(2))$  : cela r\'esulte d'un th\'eor\`eme de K. Sato et S. Saito  (\cite[Thm. 2.13]{satosaito}, \cite[Thm. 3.16]{CT15bbki}) appliqu\'e au sch\'ema $S\times_{\cl{k}}\cl{k}[[t]]$.

		Du diagramme ci-dessus on d\'eduit donc une suite exacte
		$$ 0 \to A_0(S_L)\set{\ell} \to  B  \to  H^3_{\on{nr}}(\cl{k}(X)/\cl{k},\Q_{\ell}/\Z_{\ell}(2)) \to 0.$$
		o\`u
		$$ B:= {\rm Ker} [H^1(L,H^3(\cl{S},\Z_{\ell}(2))_{\on{tors}}) \to \prod_{x\in \cl{C}^{(1)}} H^1(L_{x}, H^3(\cl{S},\Z_{\ell}(2))_{\on{tors}})].$$

		Il nous reste \`a identifier ce groupe $B$.
		Comme $L$ contient $\cl{k}$, les actions des groupes $\on{Gal}(\cl{K}/L)$ et $\on{Gal}(\cl{L}_{x}/L_{x})$ sur le module $H^1(L,H^3(\cl{S},\Z_{\ell}(2))_{\on{tors}})$ sont triviales.	Pour tout $n\geq 1$, il est clair que \[H^1(\cl{C},\Z/\ell^n)\subseteq \on{Ker}[H^1(L,\Z/\ell^n)\to   \prod_{x\in \cl{C}^{(1)}}  H^1(L_x,\Z/\ell^n)].\]  Par ailleurs, on a une suite exacte
		\[0\to H^1(\cl{C},\Z/\ell^n)\to H^1(L,\Z/\ell^n)\to \prod_{x\in \cl{C}^{(1)}}H^0(\cl{k}(x),\mu_{{\ell}^n}^{\otimes(-1)});\] voir \cite[(3.7)]{colliot1995birational}.  Pour tout $x\in \cl{C}^{(1)}$, le r\'esidu $H^1(L,\Z/\ell^n)\to H^0(\cl{k}(x),\mu_{{\ell}^n}^{\otimes(-1)})$ se  factorise comme 
		\[H^1(L,\Z/\ell^n)\to H^1(L_{x},\Z/\ell^n) \to H^0(\cl{k}(x),\mu_{{\ell}^n}^{\otimes(-1)}),\] et donc \[\on{Ker}[H^1(L,\Z/\ell^n)\to   \prod_{x\in \cl{C}^{(1)}}  H^1(L_x,\Z/\ell^n)]= H^1(\cl{C},\Z/\ell^n).\]
		On en d\'eduit que le noyau de \[H^1(L,H^3(\cl{S},\Z_{\ell}(2))_{\on{tors}})\to \prod_{x\in \cl{C}^{(1)}} H^1(L_{x}, H^3(\cl{S},\Z_{\ell}(2))_{\on{tors}})\] co\"{\i}ncide avec $H^1(\cl{C}, H^3(\cl{S},\Z_{\ell}(2))_{\on{tors}})$.
		
		On a une suite exacte courte \[0\to \on{Br}(\cl{X})^{\circ} \{\ell\} \to \on{Br}(\cl{X})\{\ell\}\to H^3(\cl{S},\Z_{\ell}(2))\set{\ell}\to 0,\] o\`u $\on{Br}(\cl{X})^{\circ}\set{\ell}\simeq (\Q_{\ell}/\Z_{\ell})^{b_2-\rho}$ est le sous-groupe divisible maximal de $\on{Br}(\cl{X})\{\ell\}$; voir \cite[(8.9)]{grothendieck1968brauer3} ou \cite[Prop. 5.2.9]{ctskobrauer}. Comme $b_{1}=0$, on a $\on{Pic}(\cl{S}) = \on{NS}(\cl{S})$. D'apr\`es \cite[(8.12)]{grothendieck1968brauer3} ou \cite[Prop. 5.2.10]{ctskobrauer} 
		on a donc des isomorphismes
		\[H^3(\cl{S},\Z_{\ell}(2))\set{\ell}   \simeq \on{Hom}(\on{NS}(\cl{S})\set{\ell},\Q_{\ell}/\Z_{\ell}(1))\simeq \on{Hom}(\on{Pic}(\cl{S})\set{\ell},\Q_{\ell}/\Z_{\ell}(1)).\]  
		
		Notons $M:=H^3(\cl{S},\Z_{\ell}(2))\set{\ell}$.  L'isomorphisme ci-dessus dit que le dual de Cartier $\hat{M}=\on{Hom}(M, \Q_{\ell}/\Z_{\ell}(1))$ de $M$ est le module galoisien $\on{Pic}(\cl{S})\set{\ell}$.
		L'accouplement $M \times \hat{M} \to \Q_{\ell}/\Z_{\ell}(1)$ et la dualit\'e de Poincar\'e sur la courbe $\cl{C}$ donnent
		un accouplement \'equivariant non d\'eg\'en\'er\'e de groupes ab\'eliens finis
		$$ H^1(\cl{C},\hat{M}) \times H^1(\cl{C}, M )  \to \Q_{\ell}/\Z_{\ell}.$$
		Par ailleurs,  on a un accouplement \'equivariant non d\'eg\'en\'er\'e de groupes ab\'eliens finis
		$$  \hat{M} \times H^1(\cl{C}, M )    \to H^1(\cl{C},   \Q_{\ell}/\Z_{\ell}(1) )=   J(C)(\cl{\F}) \{\ell\}.$$
		On a donc des isomorphismes de modules galoisiens finis 
		\[ B \simeq H^1(\cl{C}, H^3(\cl{S},\Z_{\ell}(2))) 
		\simeq {\rm Hom}_{\Z}(\hat{M},J(C)(\cl{\F}) \{\ell\} ) 
		\simeq {\rm Hom}_{\Z}( {\rm Pic}(\cl{S})\{\ell\}, J(C)(\cl{\F})).\]
		Ceci compl\`ete la d\'emonstration.
	\end{proof}

	\begin{proof}[Preuve du th\'eor\`eme \ref{mainthm2}]
		Comme la surface $S$  est g\'eom\'etriquement $CH_0$-triviale, 
		elle satisfait $b_{1}=0$ et $b_{2}-\rho=0$.
		Par le th\'eor\`eme \ref{a0}(a) on a donc $A_0(S_{\cl{\F}(C)})\{\ell\}^G=0$. 
		Comme la surface $S$  est g\'eom\'etriquement $CH_0$-triviale, 
		le th\'eor\`eme \ref{presquemainthm} nous permet de conclure.
	\end{proof}

	\begin{proof}[Preuve du th\'eor\`eme \ref{mainthm3}]
		Par le th\'eor\`eme  \ref{KCTK}, le groupe $H^3_{\on{nr}}(\cl{\F}(X),\Q_{\ell}/\Z_{\ell}(2))$ est  divisible.
		Comme $S$ est $CH_{0}$-triviale, ce groupe est fini \cite[Prop. 3.2]{colliot2013cycles}.	
		Il est donc nul.
		Par le th\'eor\`eme \ref{a0}(a), sous la condition (\ref{cond2}), on a  $A_0(S_{\cl{\F}(C)})\{\ell\}^G=0$.
		Le th\'eor\`eme \ref{presquemainthm} donne alors  la nullit\'e de $H^3_{\on{nr}}(\F(X),\Q_{\ell}/\Z_{\ell}(2))$. Par le th\'eor\`eme \ref{KCTK}(ii), le conoyau de l'application $CH^2(X) \otimes \Z_{\ell} \to H^4(X,\Z_{\ell}(2))$ est donc sans torsion. 
		On conclut comme dans la preuve du th\'eor\`eme \ref{mainthm1}: la surjectivit\'e de  ${CH}_0(C_{\Omega})\to {CH}_0(X_{\Omega})$ et  
		\cite[Prop. 3.23]{colliot2013cycles} entra\^{\i}nent que le conoyau de
		$CH^2(X) \otimes \Z_{\ell} \to H^4(X,\Z_{\ell}(2))$ est fini, il est donc nul.
	\end{proof}

	\begin{rmk}
		Soient $B$ une courbe projective, lisse et g\'eom\'etriquement connexe sur $\C$, $S$ une surface d'Enriques sur $\C$, et $X:=B\times S$. On a $\on{Pic}(S)_{\on{tors}} = \Z/2\Z$. Le th\'eor\`eme \ref{a0} donne alors une suite exacte courte de groupes finis
		\[0\to A_0(S_{\C(E)})\to J(B)(\C)[2]\to H^3_{\on{nr}}(\C(X)/\C,\Q/\Z(2))\to 0.\] On note que $H^3_{\on{nr}}(\C(X)/\C,\Q_{\ell}/\Z_{\ell}(2))=0$ pour tout $\ell\neq 2$.
		On a un isomorphisme naturel $J(B)(\C)[2]\simeq H^1(B,\Z/2\Z)$, et l'homomorphisme de restriction $\rho:CH^2(X)\to CH^2(S_{\C(B)})$ est surjectif. On a donc une suite exacte
		\[CH_1(X)_0\to H^1(B,\Z/2\Z)\to H^3_{\on{nr}}(\C(X)/\C,\Q/\Z(2))\to 0,\] o\`u $CH_1(X)_0:=\rho^{-1}(A_0(S_{\C(B)}))$. D'apr\`es \cite[Th\'eor\`eme 1.1]{ctvoisin}, on d\'eduit que l'homomorphisme $CH_1(X)_0\to H^1(B,\Z/2\Z)$ est surjectif si et seulement si la conjecture de Hodge enti\`ere pour les $1$-cycles   vaut pour $X$.
		
		Soit $\alpha\in H^2(S,\Z/2\Z)$ l'\'el\'ement correspondant au rev\^etement double par la surface $K3$ associ\'ee \`a $S$. On peut esp\'erer
		que la fl\`eche $CH_1(X)_0\to H^1(B,\Z/2\Z)$ soit donn\'ee par $Z\mapsto Z^*\alpha$, o\`u $Z^*:H^1(S,\Z/2\Z)\to H^1(B,\Z/2\Z)$ est l'homomorphisme associ\'e \`a la correspondance $Z:S\rightsquigarrow B$.  Ceci donnerait une d\'emonstration alternative de  l'\'enonc\'e \cite[Proposition 1.1]{benoist2018failure} de Benoist et Ottem.
		
		Sur le corps des complexes, ces auteurs \'etablissent l'existence de paires $(S,B)$ avec $B$ courbe elliptique
		et $S$ surface d'Enriques telles que la conjecture de Hodge enti\`ere pour les 1-cycles vaille, et d'autres pour lesquelles elle ne vaille pas.
		Ainsi, via \cite{ctvoisin}, suivant la paire $(S,B)$, le groupe $H^3_{\on{nr}}(\C(X)/\C,\Q_{\ell}/\Z_{\ell}(2))$  est nul ou non. 
		Par contraste,  si l'on remplace $\C$ par $\cl{\F}$, le
		corollaire
		\ref{dim3finialgclos} ci-dessous, cons\'equence d'un th\'eor\`eme de Schoen,   dit que pour un tel solide $X=B\times S$ avec $S$ surface d'Enriques,  on devrait toujours avoir 
		$H^3_{\on{nr}}(\cl{\F}(X)/\cl{\F},\Q_{\ell}/\Z_{\ell}(2))=0$ et l'application cycle $CH^2(\cl{X})\otimes \Z_{\ell} \to H^4(\cl{X},\Z_{\ell}(2))$ devrait \^{e}tre surjective.		\end{rmk}

	\begin{rmk}\label{tateclassique}
		(i) Montrons que si $X$ satisfait les hypoth\`eses du th\'eor\`eme \ref{mainthm2} l'application cycle (\ref{tate-int1}) est surjective.
		
		Comme $H^*(\cl{C},\Z_{\ell})$ est sans torsion et $H^2(\cl{C},\Z_{\ell})\simeq \Z_{\ell}(-1)$, la formule de K\"unneth en cohomologie $\ell$-adique \cite[Chap. VI, Cor. 8.13]{milneetale}
		nous donne un isomorphisme $G$-\'equivariant
		\[H^4(\cl{X},\Z_{\ell}(2))\simeq H^4(\cl{S},\Z_{\ell}(2))\oplus (H^3(\cl{S},\Z_{\ell}(2))\otimes H^1(\cl{C},\Z_{\ell}))\oplus H^2(\cl{S},\Z_{\ell}(1)).\]
		Pour tout $G$-module $M$ de type fini sur $\Z_{\ell}$, la fl\`eche naturelle de cup-produit \[\cup_M:M\otimes H^1(\cl{C},\Z_{\ell})\to H^1(\cl{C},M)\]
		est un isomorphisme, comme l'on voit ais\'ement par r\'eduction aux cas $M=\Z_{\ell}$ et $M=\Z/\ell^n$, $n\geq 1$. Si on pose $M=H^3(\cl{S},\Z_{\ell}(2))$, on obtient un isomorphisme $G$-\'equivariant
		\[H^3(\cl{S},\Z_{\ell}(2))\otimes H^1(\cl{C},\Z_{\ell})\simeq H^1(\cl{C},H^3(\cl{S},\Z_{\ell}(2))).\]
		
		Comme $S$ est g\'eom\'etriquement $CH_0$-triviale, on sait que $b_3(\cl{S})=b_1(\cl{S})=0$, donc $H^3(\cl{S},\Z_{\ell}(2))$ est fini et, comme discut\'e \`a la fin de la preuve du th\'eor\`eme \ref{a0}, on a un isomorphisme $G$-\'equivariant
		\[H^1(\cl{C},H^3(\cl{S},\Z_{\ell}(2)))\simeq \on{Hom}_{\Z}(\on{Pic}(\cl{S})\set{\ell},J(C)(\cl{\F})).\]
		L'hypoth\`ese (\ref{cond2}) implique alors que $H^1(\cl{C},H^3(\cl{S},\Z_{\ell}(2)))^G$ est trivial, donc
		\[H^4(\cl{X},\Z_{\ell}(2))^G\simeq H^4(\cl{S},\Z_{\ell}(2))^G\oplus  H^2(\cl{S},\Z_{\ell}(1))^G.\]
		La fonctorialit\'e de l'application de cycle $\ell$-adique nous donne le carr\'e commutatif	suivant:	
		\[
		\begin{tikzcd}
			CH^2(S)\oplus CH^1(S)\arrow[r] \arrow[d] & CH^2(X) \arrow[d]\\
			H^4(\cl{S},\Z_{\ell}(2))^G\oplus  H^2(\cl{S},\Z_{\ell}(1))^G \arrow[r,"\sim"] & H^4(\cl{X},\Z_{\ell}(2))^G.
		\end{tikzcd}
		\]	
		Comme $\F$ est fini, $S$ admet un z\'ero-cycle de degree $1$ (Lang--Weil), c'est \`a dire, la fl\`eche de degr\'e
		$CH^2(S)\otimes\Z_{\ell}\to H^4(\cl{S},\Z_{\ell}(2))^G=\Z_{\ell}$
		est surjective.	Comme $S$ est g\'eom\'etriquement $CH_0$-triviale, la fl\`eche $CH^1(S)\otimes\Z_{\ell}\to H^2(\cl{S},\Z_{\ell}(1))^G$ est aussi surjective. On conclut alors que (\ref{tate-int1})
		est surjective, comme voulu. 
		
		(ii) Nous avons le diagramme commutatif suivant
		\[
		\begin{tikzcd}
			&& CH^2(X)\otimes_{\Z}\Z_{\ell} \arrow[dr,"\text{(\ref{tate-int1})}"]  \arrow[d,"\text{(\ref{tate-int3})}"] \\
			0 \arrow[r] & H^1(\F,H^3({\cl{X},\Z_{\ell}(2)})) \arrow[r] & H^4(X,\Z_{\ell}(2)) \arrow[r] & H^4(\cl{X},\Z_{\ell}(2))^G \arrow[r] & 0,
		\end{tikzcd}
		\]	
		o\`u la suite exacte courte	vient de la suite de Hochschild-Serre en cohomologie \'etale. Le groupe fini $H^1(\F,H^3({\cl{X},\Z_{\ell}(2)}))$ est en g\'en\'eral non nul; voir la remarque \ref{remarquedim2} (i). Donc la surjectivit\'e de (\ref{tate-int1}) n'entra\^{\i}ne pas a priori celle de (\ref{tate-int3}). Nous esp\'erons revenir sur ce point dans une publication ult\'erieure.
	\end{rmk}

	\section{Applications ``classe de cycle'' en cohomologie $\ell$-adique}\label{cyclesTate}

	On donne ici  des rappels de r\'esultats que l'on peut trouver
	pour l'essentiel dans un article de B. Kahn et du premier auteur \cite{colliot2013cycles}.
	Comme d\'ej\`a  indiqu\'e, le th\'eor\`eme  \ref{KCTK}  et le corollaire \ref{dim3finialgclos} sont utilis\'es dans la
	d\'emonstration du th\'eor\`eme \ref{mainthm3} 
	dans la fin de la d\'emonstration du th\'eor\`eme \ref{mainthm3}.

	Soit $k$ un corps de caract\'eristique $p \geq 0$. Soit $\ell$ un nombre premier diff\'erent de $p$.
	Si $k$ est un corps fini, ou un corps
	alg\'ebriquement clos, situations auxquelles on va se restreindre dans ce paragraphe,  
	pour toute $k$-vari\'et\'e $X$
	les groupes de cohomomogie \'etale $H^{i}(X,\mu_{\ell^n}^{\otimes j})$
	sont finis et les groupes de cohomologie $\ell$-adiques
	$H^{i}(X,\Z_{\ell}(j)) =  \varprojlim_n H^{i}(X,\mu_{\ell^n}^{\otimes j})$
	sont des $\Z_{\ell}$-modules de type fini.
	
	On a des applications cycles
	$$CH^{i}(X)/\ell^n \to H^{2i}(X,\mu_{\ell^n}^{\otimes i})$$
	et  des applications induites de $\Z_{\ell}$-modules
	$$ \varprojlim_{n} CH^{i}(X)/\ell^n  \to H^{2i}(X,\Z_{\ell}(i)).$$
	
	On a les applications compos\'ees de $\Z_{\ell}$-modules
	$$ CH^{i}(X)  \otimes \Z_{\ell} \to  \varprojlim_{n} CH^{i}(X)/\ell^n  \to H^{2i}(X,\Z_{\ell}(i)).$$
	En utilisant le fait que les $H^{2i}(X,\Z_{\ell}(i))$ sont des $\Z_{\ell}$-modules de type fini,
	on voit que l'application $ \varprojlim CH^{i}(X)/\ell^n  \to H^{2i}(X,\Z_{\ell}(i))$ est surjective
	si et seulement si l'application compos\'ee ci-dessus est surjective.
	
	Pour toute  vari\'et\'e projective et lisse g\'eom\'e\-tri\-quement int\`egre $X$ sur un corps fini $\F$ et tous entiers $i\geq 0$,
	J. Tate  \cite{tateconj}  a conjectur\'e que les applications
	cycle rationnelles
	$$ CH^{i}(X)  \otimes \Q_{\ell} \to H^{2i}(X,\Q_{\ell}(i))$$ 
	sont surjectives. Pour $i=1$, cet \'enonc\'e est \'equivalent \`a la surjectivit\'e de
	la fl\`eche d\'eduite de la fl\`eche  de Kummer
	$$ \on{Pic}(X)  \otimes \Z_{\ell} \to H^{2}(X,\Z_{\ell}(1)),$$
	et ceci est \'equivalent \`a la finitude de $\on{Br}(X)\{\ell\}$. 
	
	On s'int\'eresse ici \`a la validit\'e de la conjecture suivante pour certaines classes
	de vari\'et\'es projectives et lisses sur un corps fini.
	
	\begin{conj}\label{cycle0}
		Soient $\F$ un corps fini et  $X$ une vari\'et\'e projective et lisse g\'eom\'e\-tri\-quement int\`egre
		sur $\F$ de dimension $d$.
		Alors, 
		pour tout $\ell$ premier distinct de la caract\'eristique de $\F$,
		l'application  cycle
		$$CH^{d-1}(X) \otimes \Z_{\ell} \to  H^{2d-2}(X,\Z_{\ell}(d-1))$$
		est surjective.
	\end{conj}

	\begin{prop}\label{coktatefini} 
		Soit $X/\F$ projective, lisse, g\'eom\'etriquement connexe de dimension $d$.
		Soit $\ell$ premier, $\ell\neq p$.
		Supposons que le groupe $\on{Br}(X)\{\ell\}$ 
		est
		fini. \begin{comment}ce qui \'equivaut au fait que  l'application cycle
		$$ CH^1(X) \otimes \Z_{\ell} \to H^2(X,\Z_{\ell}(1))$$
		est surjective.\end{comment}
		Alors le conoyau de l'application cycle
		$$ CH^{d-1}(X) \otimes \Z_{\ell} \to H^{2d-2}(X,\Z_{\ell}(d-1))$$
		est fini.
	\end{prop}
	
	\begin{proof} 
		On peut supposer $d\geq 3$.
		Comme aucune propre valeur de Frobenius sur les groupes
		de cohomologie $H^i( \cl{X},\Q_{\ell}(j) )   $ pour $i\neq 2j$
		n'est une racine de l'unit\'e (Deligne), la suite spectrale
		de Hochschild-Serre donne que
		les applications de $\Z_{\ell}$-modules de type fini
		$$H^2(X,\Z_{\ell}(1)) \to H^2(\cl{X},\Z_{\ell}(1))^G$$
		et
		$$H^{2d-2}(X,\Z_{\ell}(d-1)) \to H^{2d-2}(\cl{X},\Z_{\ell}(d-1))^G$$
		sont surjectives \`a noyau fini  
		(cf. \cite[Th\'eor\`eme 2, p.780]{colliot1983torsion}).
		
		Si $\on{cl}(L)\in H^2(X,\Z_{\ell}(1))$ est la classe d'une section hyperplane $L$ de $X$,
		le th\'eor\`eme de Lefschetz difficile dit que 
		le cup-produit par $\on{cl}(L)^{d-2}$  d\'efinit un homomorphisme
		$G$-\'equivariant
		$H^2(\cl{X},\Z_{\ell}(1)) \to H^{2d-2}(\cl{X},\Z_{\ell}(d-1))$
		\`a noyau et conoyau fini.
		
		La combinaison de ces r\'esultats  
		donne que la fl\`eche
		$$ CH^{d-1}(X) \otimes \Z_{\ell} \to H^{2d-2}(X,\Z_{\ell}(d-1))$$
		a son conoyau fini.
	\end{proof}

	\begin{lemma}\label{cycle1} 
		Soient $\F$ un corps fini, $X \subset \P^N_{\F}$ une vari\'et\'e projective et lisse int\`egre de dimension $d$,
		$Y \subset X$ une section hyperplane, et $U \subset X$ l'ouvert compl\'ementaire.
		\begin{enumerate}[label=(\roman*)]
			\item Soit $M$ le faisceau \'etale sur $U$ associ\'e \`a un module galoisien fini sur $\F$. Alors pour tout $i \geq d+2$, on a $H^{i}(U,M)=0$.
			\item Si $d \geq 4$ et $n$ est premier \`a la caract\'eristique de $\F$, alors l'application
			$$H^{2d-4}(Y,\mu_n^{\otimes(d-1)}) \to H^{2d-2}(X,\mu_n^{\otimes d})$$
			induite par l'isomorphisme de puret\'e
			$$H^{2d-4}(Y,\mu_n^{\otimes(d-1)}) \simeq H^{2d-2}_{Y}(X,\mu_n^{\otimes d})$$
			est surjective.
		\end{enumerate}
	\end{lemma}
	
	\begin{proof}
		Comme $\on{cd}(\F)=1$, la suite spectrale
		$$H^{p}(\F, H^{q}(\overline{U},M)) \Longrightarrow H^{p+q}(U,M)$$
		donne des suites exactes courtes \[0\to H^1(\F,H^{i-1}(\cl{U},M))\to H^i(U,M)\to H^i(\cl{U},M)^G\to 0,\] o\`u $G$ est le groupe de Galois absolu de $\F$. Comme $\cl{U}$ est affine et $M$ est fini, par le th\'eor\`eme de Lefschetz affine on a $H^i(\cl{U},M)=0$ pour $i\geq d+1$. Ceci d\'emontre~(i). La partie (ii) suit de (i)	pour $i=d-1$ et $M=\mu_n^{\otimes(d-1)}$, en utilisant la suite de Gysin.
	\end{proof}

	\begin{prop}\label{dim3suffit}
		Pour \'etablir la conjecture \ref{cycle0} pour toute vari\'et\'e de dimension  $d\geq 3$, il suffit de le faire pour toute vari\'et\'e de dimension  $d=3$.
	\end{prop}
	\begin{proof}
		On proc\`ede par r\'ecurrence sur $d\geq 3$. Soit $X/\F$ une vari\'et\'e projective, lisse et g\'eom\'etriquement int\`egre de dimension $d\geq 4$. Par le th\'eor\`eme de Bertini sur les corps finis \cite{poonen2004bertini}, il existe une immersion ferm\'ee $X \hookrightarrow \P^N$ et une section hyperplane $Y$ de $X$ qui est lisse et g\'eom\'etriquement int\`egre de dimension $d-1$. Pour tout $n\geq 1$, on a un carr\'e commutatif
		\[
		\begin{tikzcd}
			CH^{d-2}(Y)/\ell^n \arrow[r] \arrow[d] & H^{2d-4}(Y,\mu_{\ell^n}^{\otimes (d-2)}) \arrow[d] \\
			CH^{d-1}(X)/\ell^n \arrow[r] & H^{2d-2}(X,\mu_{\ell^n}^{\otimes(d-1)}).
		\end{tikzcd}		
		\]
		On passe \`a la limite projective sur $n$. On obtient le diagramme commutatif suivant:
		\[
		\begin{tikzcd}
			\varprojlim_{n} CH^{d-2}(Y)/\ell^n \arrow[r] \arrow[d] & H^{2d-4}(Y,\Z_{\ell}(d-2)) \arrow[d] \\
			\varprojlim_{n} CH^{d-1}(X)/\ell^n \arrow[r] & H^{2d-2}(X,\Z_{\ell}(d-1)).
		\end{tikzcd}		
		\]
		Par l'hypoth\`ese de r\'ecurrence, la fl\`eche horizontale en haut est surjective. D'apr\`es le lemme \ref{cycle1}(ii), la fl\`eche verticale \`a droite est surjective. On conclut que la fl\`eche horizontale en bas est surjective. 
	\end{proof}
	
	\begin{rmk} Pour \'etablir la conjecture 
		\ref{cycle0} pour une vari\'et\'e $X$ avec un plongement $X \subset \P^n_{\F}$ donn\'e, par un argument de normes, il suffit de l'\'etablir sur des extensions finies $\F'/\F$ de degr\'es premiers
		entre eux. On peut donc se contenter d'utiliser le th\'eor\`eme de Bertini sur les corps finis ``suffisamment'' gros,
		et dans l'argument ci-dessus de  prendre les sections hyperplanes pour le plongement donn\'e. Ainsi, pour \'etablir la conjecture \ref{cycle0} pour
		les hypersurfaces cubiques lisses dans $\P^n_{\F}$ pour $n \geq 4$, il suffit de l'\'etablir  pour les hypersurfaces cubiques lisses
		dans $\P^4_{\F'}$ pour tout corps fini $\F'$. En caract\'eristique diff\'erente de $2$, ceci est connu  (voir \cite[Thm. 5.1]{cttunis}).
	\end{rmk}
	
	\begin{rmk}
		Supposons vraie la conjecture de Tate pour les surfaces sur un corps fini. Pour d\'emontrer la conjecture \ref{cycle0}, il suffirait de montrer :
		pour $X$ de dimension 3  et  toute classe $\xi$ dans $H^4(X,\Z_{\ell}(2))$ il existe une section hyperplane lisse
		$Y \subset X$ telle que $\xi$ soit support\'ee sur $Y$,  c'est-\`a-dire  telle que la restriction de $\xi$
		dans $H^4(U,\Z_{\ell}(2))=H^1(\F, H^3({\overline U},\Z_{\ell}(2)))$ soit nulle.
	\end{rmk}

	Soit $G$ un groupe profini. Pour tout $\Z_{\ell}$-module   $M$ \'equip\'e d'une action continue
	de $G$, on note $M^{(1)} \subset M$ le sous-module form\'e des \'el\'ements dont le stabilisateur
	est un sous-groupe ouvert de $G$.  
	
	\begin{lemma}\label{CTK41} \cite[Lemme 4.1]{colliot2013cycles}
		Si $G$ est un groupe profini et $M$ est un $\Z_{\ell}$-module $M$  de type fini muni d'une action continue de $G$, alors
		le quotient  $M/M^{(1)}$ est sans torsion.
	\end{lemma}

	Soient $\F$ un corps fini, $\cl{\F}$ une cl\^{o}ture alg\'ebrique de $\F$
	et $G:=\on{Gal}(\cl{\F}/\F)$. Soit $X/\F$ une vari\'et\'e projective, lisse, g\'eom\'etriquement connexe
	de dimension $d$. Soit $\cl{X}=X\times_{\F}\cl{\F}$.
	Pour tout entier $i \geq 0$, l'application cycle
	\[CH^{i}(\cl{X}) \otimes \Z_{\ell} \to H^{2i}(\cl{X},\Z_{\ell}(i))\]
	a son image dans le sous-groupe $H^{2i}(\cl{X},\Z_{\ell}(i))^{(1)}$.
	
	\begin{thm}\label{schoen}(Schoen)
		Supposons vraie la conjecture de Tate pour
		les surfaces sur les corps finis. 
		Alors pour toute vari\'et\'e $X/\F$ projective, lisse, g\'eom\'etriquement con\-ne\-xe,
		de dimension $d$, l'image de l'application cycle
		\[CH^{d-1}(\cl{X}) \otimes \Z_{\ell} \to H^{2d-2}(\cl{X},\Z_{\ell}(i))\]
		est   le sous-groupe de $H^{2d-2}(\cl{X},\Z_{\ell}(i))^{(1)}$  form\'e des \'el\'ements dont le stabilisateur
		est un sous-groupe ouvert de G.  
	\end{thm}
	\begin{proof} 
		Voir \cite[Theorem 0.5]{schoen1998integral} et \cite{colliot2010autour}.
	\end{proof}

	\begin{cor}\label{corschoen}\
		Soient $k$ une cl\^oture alg\'e\-brique
		d'un corps fini de caract\'eristique $p$ et $X$ une $k$-vari\'et\'e projective et lisse, g\'eom\'e\-tri\-quement
		connexe, de dimension $d$. 
		Soit $\ell$ un premier distinct de $p$.
		Si  la conjecture de Tate vaut pour toutes 
		les surfaces sur les corps finis, alors :
		
		(i) 
		Le conoyau de l'application cycle
		\[CH^{d-1}(X) \otimes \Z_{\ell} \to H^{2d-2}(X,\Z_{\ell}(d-1))\]
		est un $\Z_{\ell}$-module de type fini sans torsion.
		
		(ii) Si 	$\on{Br}(X)\{\ell\}$ est fini, l'application cycle
		\[CH^{d-1}(X) \otimes \Z_{\ell} \to H^{2d-2}(X,\Z_{\ell}(d-1))\]
		est surjective.
	\end{cor}
	
	\begin{proof} 
		Comme  expliqu\'e dans la d\'emonstration de \cite[Prop. 4.2]{colliot2013cycles}, la combinaison
		du th\'eor\`eme \ref{schoen} et du lemme \ref{CTK41} donne  l'\'enonc\'e (i).
		
		Montrons (ii). On proc\`ede comme \`a la proposition \ref{coktatefini}.
		Sous l'hypoth\`ese de finitude du groupe de Brauer,
		le module de Tate $T_{\ell}(\on{Br}(X))$ est  nul, et
		l'application cycle
		\[CH^1(X) \otimes \Z_{\ell} \xrightarrow{\on{cl}} H^2(X,\Z_{\ell}(1))\]
		est surjective; voir \cite[(8.7)]{grothendieck1968brauer3}. Soit $L\in CH^1(X)$ la classe d'un diviseur ample. On a donc un diagramme commutatif
		\[
		\begin{tikzcd}
			CH^1(X)\otimes \Z_{\ell} \arrow[r,->>,"\on{cl}"] \arrow[d, "\cap L^{d-2}"] & H^2(X,\Z_{\ell}(1)) \arrow[d, "\cap \on{cl}(L)^{d-2}"]  \\
			CH^{d-1}(X)\otimes \Z_{\ell} \arrow[r,"\on{cl}"] & H^{2d-2}(X,\Z_{\ell}(d-1)).
		\end{tikzcd}
		\]
		D'apr\`es le th\'eor\`eme de Lefschetz difficile, l'homomorphisme vertical de droite devient un isomorphisme apr\`es tensorisation par $\Q_{\ell}$. Il s'ensuit que la fl\`eche
		\[CH^{d-1}(X)\otimes  \Z_{\ell} \to H^{2d-2}(X, \Z_{\ell}(d-1))\]
		a un conoyau fini.	L'\'enonc\'e (i) assure alors la nullit\'e de ce conoyau.	
	\end{proof}
	
	\begin{thm}\label{KCTK} \cite[Thm. 1.1]{kahn2012classes}, \cite [Thm. 2.2]{colliot2013cycles} 
		Soit $k$ un corps fini ou un corps alg\'ebriquement clos.
		Soit $X$ une $k$-vari\'et\'e projective, lisse, g\'eom\'etriquement con\-ne\-xe.
		Soit $k(X)$ son corps des fonctions rationnelles.
		Soit  $\ell$ premier distinct de la caract\'eristique de $k$. 
		Les deux groupes suivants  sont finis et sont isomorphes 
		entre eux :
		
		(i)  Le quotient du groupe $H^3_{\on{nr}}(k(X),\Q_{\ell}/\Z_{\ell}(2))$
		par son sous-groupe divisible maximal.
		
		(ii) Le sous-groupe de torsion du $\Z_{\ell}$-module de type fini
		conoyau de l'application cycle 
		\[ CH^2(X) \otimes \Z_{\ell} \to H^{4}(X,\Z_{\ell}(2)).\]  
	\end{thm}
	
	En dimension 3, la combinaison de ce r\'esultat et des corollaires du th\'eor\`eme de Schoen donne :

	\begin{cor}\label{dim3finialgclos}\cite[Prop. 4.2, Prop. 3.2]{colliot2013cycles}
		Soient $k$ une cl\^oture alg\'ebrique
		d'un corps fini de caract\'eristique $p$ et $X$ une $k$-vari\'et\'e projective et lisse, g\'eom\'e\-tri\-quement
		connexe, de dimension $3$. 
		Soit $\ell$ un premier distinct de $p$.
		Supposons vraie la conjecture de Tate pour
		les surfaces sur les corps finis.  Alors :
		
		(i) Le groupe $H^3_{\on{nr}}(k(X),\Q_{\ell}/\Z_{\ell}(2))$
		est un groupe divisible. 	 
		
		(ii) Si de plus il existe une surface projective et lisse $S/k$,  et un $k$-morphisme $S \to X$
		tel que pour tout corps alg\'ebriquement clos $\Omega$ contenant $k$
		l'application induite $CH_{0}(S_{\Omega}) \to CH_{0}(X_{\Omega})$ est surjective,
		alors $H^3_{\on{nr}}(k(X),\Q_{\ell}/\Z_{\ell}(2))=0$.
	\end{cor}
	
	\begin{proof}  D'apr\`es le th\'eor\`eme \ref{KCTK},  le quotient 
		de $H^3_{\on{nr}}(k(X),\Q_{\ell}/\Z_{\ell}(2))$ par son sous-groupe divisible maximal
		s'identifie au sous-groupe de  torsion du conoyau de
		\[CH^{2}(X)\otimes  \Z_{\ell} \to H^4(X, \Z_{\ell}(2)).\]
		Mais sous l'hypoth\`ese sur la conjecture de Tate, le corollaire \ref{corschoen}
		assure que ce conoyau n'a pas de torsion. Ceci \'etablit (i).
		
		Dans la situation de (ii), l'hypoth\`ese sur les groupes de Chow de dimension z\'ero
		et un argument de correspondances bien connu  \cite[Prop. 3.2]{colliot2013cycles}
		implique que le groupe $H^3_{\on{nr}}(k(X),\Q_{\ell}/\Z_{\ell}(2))$ est annul\'e par
		un entier $N>0$. Comme le groupe  $H^3_{\on{nr}}(k(X),\Q_{\ell}/\Z_{\ell}(2))$ est divisible, il est nul. 
	\end{proof}

	\begin{rmk}	
		Sur le corps $\C$ des complexes,  soit $X=E\times S$ le produit
		d'une courbe elliptique $E$ et d'une surface d'Enriques.
		Pour $s\in S(\C)$ fix\'e, l'inclusion $E \to X$ donn\'ee par $m \mapsto (m,s)$
		satisfait que $CH_{0}(C_{\Omega}) \to CH_{0}(X_{\Omega})$ est surjectif
		pour tout corps alg\'ebriquement clos $\Omega$ contenant $\C$.
		Ceci implique que le groupe $\on{Br}(X)$ est d'exposant fini, et donc est fini.
		Mais il existe de tels couples $(E,S)$ pour lesquels
		le groupe $H^3_{\on{nr}}(\C(X),\Q_{2}/\Z_{2}(2))$ est non nul et   
		la conjecture de Hodge enti\`ere pour les cycles de codimension 2
		n'est pas satisfaite \cite{benoist2018failure, colliot2019cohomologie}. La situation sur 
		le corps des complexes est donc 
		diff\'erente de celle des corollaires (conditionnels) \ref{corschoen} et \ref{dim3finialgclos}. Voir aussi \cite{colliot2010autour}.
	\end{rmk}

	\section{Remerciements}
	Federico Scavia b\'en\'eficie  d'une bourse d'\'etudes de University of British Columbia.
	Il remercie le D\'epartement de Math\'ematiques d'Orsay (Universit\'e Paris Saclay) pour son hospitalit\'e pendant l'automne 2019.
	Cette recherche a \'et\'e rendue possible gr\^ace au financement qui lui a \'et\'e fourni par Mitacs.

\end{document}